\documentclass[11pt,a4paper]{amsart}
\newcommand{\Title}{Distances Between von Neumann Subalgebras: Spin Models, Commuting Squares, and Free Group Factors}%
\newcommand{\ShortTitle}{\Title}%

\newcommand{\AuthorOne}{Indrajit Ghosh}%
\newcommand{\AuthorOneAddr}{%
	Department of Mathematics, Indian Institute of Technology Kanpur, Uttar Pradesh 208 016, India
}%

\newcommand{\AuthorOneEmail}{%
	indrajitghosh912@gmail.com, indrajitg@iitk.ac.in
}%

\newcommand{\AuthorTwo}{Sumit Kumar}%
\newcommand{\AuthorTwoAddr}{Department of Mathematics, Indian Institute of Technology Kanpur, Uttar Pradesh 208 016, India
}%
\newcommand{\AuthorTwoEmail}{sumitkumar.sk809@gmail.com}

\newcommand{\SubjectClassText}{Primary 46L05, 47C15, 47L40; Secondary 46L10}

\newcommand{\Keywords}{Kadison-Kastler Distance, Mashood-Taylor Distance, Spin Model Subfactor, Group von Neumann Algebras, Commuting Square, Hadamard Matrices}

\newcommand{\pdfTitle}{\Title}
\newcommand{\pdfAuthor}{Indrajit Ghosh}
\newcommand{\pdfSubject}{Mathematics, Research paper}
\newcommand{\pdfKeywords}{\Keywords}
\newcommand{\pdfCreator}{TeXlive}
\newcommand{\pdfCreationDate}{\today}
\newcommand{\pdfColorLink}{true}
\newcommand{\pdfLinkColor}{cyan}
\newcommand{\pdfUrlColor}{blue}
\newcommand{\pdfCiteColor}{magenta}

\usepackage[top=0.9in, bottom=1in, left=0.7in, right=0.7in]{geometry}
\usepackage{amsmath, amssymb, amsthm} 
\usepackage[utf8]{inputenc}
\usepackage[T1]{fontenc}
\usepackage{mathtools}
\usepackage{mathrsfs} 
\usepackage{xfrac} 
\usepackage{dsfont} 
\usepackage{array}
\usepackage{verbatim}
\usepackage{graphicx}
\usepackage{mdframed}
\usepackage{enumitem} 
\usepackage{hyperref}
\hypersetup{
	pdftitle={\pdfTitle},
	pdfauthor={\pdfAuthor},
	pdfsubject={\pdfSubject},
	pdfcreationdate={\pdfCreationDate},
	pdfcreator={\pdfCreator},
	pdfkeywords={\pdfKeywords},
	colorlinks=\pdfColorLink,
	linkcolor={\pdfLinkColor},
	urlcolor=\pdfUrlColor,
	citecolor=\pdfCiteColor,
	pdfpagemode=UseOutlines,
}
\usepackage{tikz-cd} 
\usepackage{lipsum}
\usetikzlibrary{matrix,arrows}
\usepackage[english]{babel}
\usepackage{lmodern}
\usepackage{bbm} 
\usepackage[dvipsnames]{xcolor}
\usepackage[most]{tcolorbox}
\usepackage{xparse}

\theoremstyle{plain}
\newtheorem{theorem}{Theorem}[section]
\newtheorem{prop}[theorem]{Proposition}
\newtheorem{lem}[theorem]{Lemma}
\newtheorem{cor}[theorem]{Corollary}

\theoremstyle{definition}
\newtheorem{definition}[theorem]{Definition}
\newtheorem{example}[theorem]{Example}

\theoremstyle{remark}
\newtheorem{remark}[theorem]{Remark}

\numberwithin{equation}{section}

\newtheoremstyle{ser}
{8pt}
{8pt}
{\it}
{}
{\sf}
{:}
{6mm}
{}

\theoremstyle{ser}

\newtheoremstyle{serr}
{8pt}
{8pt}
{\normalfont}
{}
{\sf}
{.}
{6mm}
{}

\theoremstyle{serr}

\theoremstyle{ser}

\theoremstyle{ser}

\newtheoremstyle{collabquestion}
  {8pt}
  {8pt}
  {\normalfont}
  {}
  {\sffamily\bfseries\color{blue!70!black}}
  {.}
  {.5em}
  {}

\theoremstyle{collabquestion}
\newtheorem{qninner}{Question}

\definecolor{indraRed}{rgb}{0.593, 0.183, 0.183}
\definecolor{indraPink}{rgb}{0.858, 0.188, 0.478}
\definecolor{indraBlue}{rgb}{0, 0.199, 0.398}
\definecolor{madridBlue}{rgb}{0.199, 0.199, 0.695}
\definecolor{metropolisThemeColor}{rgb}{0.105, 0.214, 0.234}
\definecolor{metropolisBarColor}{rgb}{0.984, 0.0.515, 0.015}
\definecolor{UBCblue}{rgb}{0.04706, 0.13725, 0.26667} 
\definecolor{UBCgrey}{rgb}{0.3686, 0.5255, 0.6235} 

\makeatletter
\def\mathcolor#1#{\@mathcolor{#1}}
\def\@mathcolor#1#2#3{%
	\protect\leavevmode
	\begingroup
	\color#1{#2}#3%
	\endgroup
}
\makeatother

\makeatletter
\def\ps@headings{\ps@empty
  \def\@evenhead{\normalfont\scriptsize\hfil \leftmark{}{}\hfil}%
  \def\@oddhead{\normalfont\scriptsize\hfil \rightmark{}{}\hfil}%
  \let\@mkboth\markboth
  \def\@evenfoot{\normalfont\scriptsize\hfil\thepage\hfil}%
  \def\@oddfoot{\normalfont\scriptsize\hfil\thepage\hfil}%
}
\makeatother

\makeatletter
\def\author@andify{%
  \nxandlist {\unskip ,\penalty-1 \space\ignorespaces}%
    {\unskip {} \@@and~}%
    {\unskip \penalty-2 \space \@@and~}%
}
\makeatother

\definecolor{indraRed}{rgb}{0.593, 0.183, 0.183}
\definecolor{indraPink}{rgb}{0.858, 0.188, 0.478}
\definecolor{indraBlue}{rgb}{0, 0.199, 0.398}
\definecolor{madridBlue}{rgb}{0.199, 0.199, 0.695}
\definecolor{metropolisThemeColor}{rgb}{0.105, 0.214, 0.234}
\definecolor{metropolisBarColor}{rgb}{0.984, 0.0.515, 0.015}
\definecolor{UBCblue}{rgb}{0.04706, 0.13725, 0.26667} 
\definecolor{UBCgrey}{rgb}{0.3686, 0.5255, 0.6235} 

\newcommand{\spn}{{\operatorname{span}\,}}

\newcommand{\C}{\mathbb{C}}
\newcommand{\R}{\mathbb{R}}
\newcommand{\N}{\mathbb{N}}

\newcommand{\Z}{\mathbb{Z}}

\renewcommand{\emptyset}{\varnothing}

\newcommand{\Bh}{\mathcal{B}(\mathcal{H})} 
\NewDocumentCommand{\mn}{m o}
{
  \mathbb{M}_{#1}\IfValueT{#2}{(#2)}
}

\NewDocumentCommand{\hn}{m o}
{
  \mathbb{H}^u_{#1}\IfValueT{#2}{(#2)}
}

\DeclareMathOperator{\dist}{dist} 

     \newcommand{\sA}{\mathcal A}		
     \newcommand{\sB}{\mathcal B}		
     \newcommand{\sC}{\mathcal C}		
     \newcommand{\sD}{\mathcal D}		
     		
\newcommand{\fF}{\mathfrak F}     		
     \newcommand{\sG}{\mathcal G}		
     \newcommand{\sH}{\mathcal H}

     \newcommand{\sM}{\mathcal M}		\newcommand{\kM}{\mathscr{M}}
     \newcommand{\sN}{\mathcal N}		
     		
     \newcommand{\sP}{\mathcal P}		
     \newcommand{\sQ}{\mathcal Q}		
     \newcommand{\sR}{\mathcal R}		\newcommand{\kR}{\mathscr{R}}
     \newcommand{\sS}{\mathcal S}		
     		
     \newcommand{\sU}{\mathcal U}

     \newcommand{\sX}{\mathcal X}		
     \newcommand{\sY}{\mathcal Y}		
     \newcommand{\sZ}{\mathcal Z}		

\newcommand{\mF}{\mathbb{F}}
\newcommand{\mI}{\mathbb{I}}

\newcommand{\sfH}{\mathsf{H}}

\NewDocumentCommand{\nor}{g}
  {\sN\IfValueT{#1}{_{#1}}}

\NewDocumentCommand{\unor}{g}
  {\sU\sN\IfValueT{#1}{_{#1}}}

\NewDocumentCommand{\grnor}{g}
  {\sG\sN\IfValueT{#1}{_{#1}}}

\newcommand{\dkk}{\mathrm{d}_{\mathrm{KK}}}
\newcommand{\ball}[1]{\left( #1\right)_1}
\newcommand{\dmt}{\mathrm{d}_{\mathrm{MT}}}
\newcommand{\gen}[1]{\left\langle #1 \right\rangle}

\newcommand{\op}[1]{\left\| #1 \right\|_{\mathrm{op}}}
\newcommand{\Deltan}[1][n]{\Delta^{(#1)}}

\newcommand{\II}{\mathrm{II}}

\newcommand{\bg}{\sim}
\newcommand{\hada}{\sim_{\mathrm{h}}}
\newcommand{\fourier}[1]{\mathsf{F}_{#1}}

\begin{document}
	
        \title[\ShortTitle]{\MakeUppercase\Title}
	\author{\AuthorOne}
	\address[Indrajit Ghosh]{\AuthorOneAddr}
	
	\email{\AuthorOneEmail}

    \author{\AuthorTwo}
	\address[Sumit Kumar]{\AuthorTwoAddr}
	\email{\AuthorTwoEmail}

	\date{}
	\subjclass{\SubjectClassText}
	
	
	\begin{abstract}
		We investigate the relative position of von Neumann subalgebras by studying their Mashood--Taylor ($\mathrm{d}_{\mathrm{MT}}$) and Kadison--Kastler ($\mathrm{d}_{\mathrm{KK}}$) distances, alongside the interior angle. For each $n\in\mathbb{N}$, we show that the hyperfinite $\mathrm{II}_1$-factor $\mathscr{R}$ contains an uncountable family of pairwise distinct regular $n\times n$ spin model subfactors. In particular, this yields a continuous family $(\mathscr{R}_{\mathsf{H}_\alpha})_{\alpha\in[0,\pi)}$ of $2\times2$ spin model subfactors. Moreover, equipped with the Mashood--Taylor metric, this family is homeomorphic to the circle group $\mathbb{T}$. More precisely, we establish the exact formula
\[
\mathrm{d}_{\mathrm{MT}}(\mathscr{R}_{\mathsf{H}_\alpha},\mathscr{R}_{\mathsf{H}_\beta})
= \vert\sin(\alpha-\beta)\vert.
\]
Furthermore, we prove that a pair of such spin model subfactors forms a commuting square over their intersection if and only if they are maximally distant ($\mathrm{d}_{\mathrm{MT}}=1$). Motivated by this, we establish general structural results showing that commuting squares of $\mathrm{II}_1$-factors (under natural index conditions) force maximal distance, yielding $\mathrm{d}_{\mathrm{KK}}=1=\mathrm{d}_{\mathrm{MT}}$. We also show that two diffuse subalgebras that are orthogonal in the sense of Popa must be maximally distant. In contrast, we observe that no two subfactors in the family $(\mathscr{R}_{\mathsf{H}_\alpha})_{\alpha\in[0,\pi)}$ are orthogonal in the sense of Popa. Nevertheless, whenever two such spin model subfactors are maximally distant, their interior angle over their intersection is $\pi/2$, showing that ``orthogonality'' in terms of the interior angle is distinct from orthogonality in the sense of Popa. Finally, in the free group factor $L(\mathbb{F}_2)=L(\langle a,b\rangle)$, we establish the explicit formula
\[
\mathrm{d}_{\mathrm{MT}}(L(\langle a\rangle),uL(\langle a\rangle)u^*)
= \sqrt{1-\vert\tau(u)\vert^4}
\]
for $u\in L(\langle b\rangle)$. We also construct maximally distant masas in $L(\mathbb{F}_2)$ that do not arise from subgroups of $\mathbb{F}_2$.
	\end{abstract}

	\maketitle%
        \thispagestyle{empty}%


    \section{Introduction}

Quantitative measures of the distance between von Neumann subalgebras provide essential insights into their relative positions and algebraic structures. Two prominently studied metrics in this context are the Kadison-Kastler distance \cite{Kadison_Kastler}, denoted by $\dkk$, and the Mashood-Taylor distance \cite{Mashood-Taylor-1988}, denoted by $\dmt$ (see also \cite{christensen-79}). 
Understanding how these geometric distances constrain, or are constrained by, the algebraic configurations of the subalgebras remains a central theme in the theory of operator algebras. In our recent work \cite{ghoshkumar-kk2026}, we investigated the Kadison-Kastler distance and the interior angle (in the sense of \cite{Bakshi_gupta_2021}; see also \cite{bakshi-etal-2019}) between maximal abelian subalgebras (masas) in $\mn{n}$. A principal conclusion of that study was an explicit relationship between the Kadison-Kastler distance and the interior angle. This naturally motivates the present investigation, in which we explore analogous geometric relationships for the Mashood-Taylor distance across various classes of subalgebras.

To pursue this, we first focus on the spin model subfactors of the hyperfinite $\II_1$-factor $\mathscr{R}$, originally introduced by Jones (\cite{jones-2021}; see also \cite{jones-sunder}). Given $n\times n$ Hadamard unitaries $u,v\in\hn{n}$ that are inequivalent in the sense of Bakshi and Guin \cite{BG2025-1}, one obtains distinct spin model subfactors $\kR_u$ and $\kR_v$ of $\mathscr{R}$. In fact, these spin model subfactors give rise to a rich collection of regular subfactors of $\mathscr{R}$: for every $n\in\mathbb{N}$, $\mathscr{R}$ contains an uncountable family of pairwise distinct \emph{regular} $n\times n$ spin model subfactors (see Proposition \ref{prop:uncountable-reg-spin-models}). The first part of this paper is devoted to the relative position and distances between such subfactors. The structure is particularly tractable in the $2\times 2$ case, where Hadamard equivalence classes can be parametrized by a single variable $\alpha\in[0,\pi)$, yielding a continuous family of subfactors $(\kR_{\sfH_\alpha})_{\alpha\in[0,\pi)}$. In fact, if we denote this family by
\[
\mathrm{Sp}_2:=\{\kR_{\sfH_\alpha}:\alpha\in[0,\pi)\},
\]
then, equipped with the Mashood--Taylor metric $\dmt$, the space $\mathrm{Sp}_2$ is homeomorphic to the circle group $\mathbb{T}$:
\vspace{0.3cm}

\noindent \textbf{Theorem A:} (Theorem \ref{thm:sp-2-is-circle})
The space $\mathrm{Sp}_2$ of all $2\times2$ spin model subfactors, equipped with the metric $\dmt$, is homeomorphic to the circle group $\mathbb{T}$:
\[
(\mathrm{Sp}_2,\dmt)\cong \mathbb{T}.
\]
In particular, $(\mathrm{Sp}_2,\dmt)$ is a compact and path connected metric space.
\vspace{0.2cm}

Thus, $\mathrm{Sp}_2$ not only has the topology of a circle, but also admits a natural topological group structure, obtained by transporting the group operation of $\mathbb{T}$ through the above identification. We then determine the Mashood--Taylor metric on this topological group explicitly. More precisely, the distance between two elements of $\mathrm{Sp}_2$ is given by the following formula:
\vspace{0.3cm}

\noindent \textbf{Theorem B:} (Theorem \ref{thm:mt-distance-2by2})
For any distinct $\alpha,\beta\in[0,\pi)$, we have
\[
\dmt(\kR_{\sfH_\alpha},\kR_{\sfH_\beta})
= \left|\sin(\alpha-\beta)\right|.
\]
An immediate consequence of this explicit formula is that the distance between such spin model subfactors is not quantized; rather, every value in $[0,1]$ occurs as the Mashood--Taylor distance between some pair in this family (see Corollary \ref{cor:dmt-is-continuous}).

The absence of quantization for these distances closely parallels the behavior of the Kadison-Kastler distance and interior angle between masas in $\mn{n}$ established in \cite{ghoshkumar-kk2026}, where the interior angle can take any value in $[0,\pi/2]$. This flexible geometric behavior stands in stark contrast to the deep rigidity found in other settings, such as the result of \cite[Theorem 3.6]{bakshi-etal-2019}, which proves that for minimal intermediate subfactors of a finite-index irreducible subfactor, the interior angle cannot fall below $\pi/3$. This dichotomy highlights a broader motivating question of our work: identifying which classes of subalgebras exhibit quantized geometric invariants and which allow for continuous variation. For the $2 \times 2$ spin model subfactors, we further show that the interior angle also attains every value in $[0,\pi/2]$, and we obtain the exact relationship (see Corollary \ref{cor:dmt-is-same-alpha-2by2}):
$$\dmt(\kR_{\sfH_\alpha},\kR_{\sfH_\beta}) = \sqrt{1-\cos\left( \alpha_{\kR_{\sfH_\alpha}\cap\kR_{\sfH_\beta}}^{\kR} (\kR_{\sfH_\alpha},\kR_{\sfH_\beta}) \right)}.$$

The intersection of any two $2 \times 2$ spin model subfactors is known to be a $\II_1$-subfactor, namely a vertex model subfactor \cite{BG2025-1}. This naturally leads to the question of when two such subfactors form a commuting square over their intersection. In this direction, we prove the following:
\vspace{0.3cm}

\noindent \textbf{Theorem C:} (Theorem \ref{cor:comm-sq-d-1-2times2}) Let $\alpha,\beta\in[0,\pi)$ be distinct. Then the following quadruple
    \[
\begin{array}{ccc}
\kR_{\sfH_\alpha}  & \subset & \kR \\
\cup && \cup \\
\kR_{\sfH_\alpha} \cap \kR_{\sfH_\beta} & \subset & \kR_{\sfH_\beta}
\end{array}
\]
is a commuting square if and only if $\dmt (\kR_{\sfH_\alpha}, \kR_{\sfH_\beta}) = 1.$
\vspace{0.3cm}

Because $\dmt \le \dkk$ in general, this demonstrates that maximal Kadison-Kastler distance is a necessary condition for such a pair of spin model subfactors to form a commuting square (see Corollary \ref{cor:comm-sq-dkk-1-2times2}). It is important to emphasize that this necessity is specific to the structural properties of these configurations; as we show in Remark \ref{rem:comm-sq-vs-max-d}, for general quadruples of subfactors, it is possible to achieve a distance of $1$ without the presence of a commuting square.

Motivated by this, we establish broader general results detailing when commuting squares of $\II_1$-factors force maximal spatial separation. More precisely, we prove the following:
\vspace{0.3cm}

\noindent \textbf{Theorem D:} (Theorem \ref{thm:dkk-is-1-for-comm-sq-spin}) Let $\sM$ be a finite von Neumann algebra, and let $\sN,\sP,\sQ$ be $\II_1$-subfactors of $\sM$ such that
\[
\begin{array}{ccc}
    \sP & \subset & \sM \\
    \cup && \cup \\
    \sN & \subset & \sQ
\end{array}
\]
is a commuting square. Suppose that the index of at least one of the inclusions $\sN\subseteq\sP$ or $\sN\subseteq\sQ$ belongs to $\{2,3\}\cup[4,\infty]$. Then $\dkk(\sP,\sQ)=1=\dmt(\sP,\sQ)$.
\vspace{0.2cm}

Additionally, we prove an analogous result for intermediate subfactors:
\vspace{0.3cm}

\noindent \textbf{Theorem E:} (Theorem \ref{thm:comm-sq-imply-d-1-ii-1}) Let $\sN \subseteq \sM$ be $\mathrm{II}_1$ subfactors with $[\mathcal{M}:\mathcal{N}]\ge 4$, and let $\sP,\sQ$ be intermediate $\mathrm{II}_1$ subfactors.
    Suppose that
    \[\begin{array}{ccc}
            \sP & \subset & \sM \\
            \cup & & \cup \\
            \sN & \subset & \sQ
        \end{array}
    \]
    is a commuting square. Then
    \[
        \dmt(\sP,\sQ)=1=\dkk(\sP,\sQ).
    \]

These theorems show that across broad classes of $\II_1$-factor configurations, a commuting square naturally forces the participating subalgebras to be maximally distant. Applying this back to the spin model setting, we got the following satisfying result in the spin model subfactors:
\vspace{0.3cm}

\noindent \textbf{Theorem F:} (Theorem \ref{thm:spin-model-n-times-n})  Let $u,v\in\hn{n}$ with $u\not\bg v$ such that $\kR_u\cap\kR_v$ is a finite-index subfactor of $\kR$.
    If
    \[
        \begin{array}{ccc}
            \kR_u & \subset & \kR \\
            \cup && \cup \\
            \kR_u\cap\kR_v & \subset & \kR_v
        \end{array}
    \]
    is a commuting square, then
    \[
        \dkk(\kR_u,\kR_v)=1=\dmt(\kR_u, \kR_v).
    \]

Consequently, if $\dkk(\kR_u,\kR_v) \neq 1$, it must be that either $\kR_u \cap \kR_v$ is not a finite-index subfactor of $\mathscr{R}$, or the quadruple fails to form a commuting square (see Corollary \ref{cor:necessity-for-commuting-square}). 

In \S\ref{subsec:regularity}, we briefly investigate the regularity of spin model subfactors of the hyperfinite $\mathrm{II}_1$-factor. In particular, we show that, for every $n\in\mathbb{N}$, $\mathscr{R}$ contains an uncountable family of pairwise distinct regular $n\times n$ spin model subfactors (see Proposition \ref{prop:uncountable-reg-spin-models}). This is obtained by constructing an uncountable family of Hadamard unitaries that are Hadamard equivalent to the Fourier matrix but pairwise inequivalent in the sense of Bakshi--Guin (see Lemma \ref{lem:uncountable-unitaries}).

In \cite{popa-1983}, Popa introduced a notion of orthogonality for two subalgebras of a finite von Neumann algebra. Namely, subalgebras $\sP,\sQ\subseteq (\kM,\tau)$ are said to be \emph{orthogonal} if $\tau(xy)=0$ whenever $x\in\sP$ and $y\in\sQ$ satisfy $\tau(x)=\tau(y)=0$. In the following theorem, we show that two diffuse subalgebras can be orthogonal only when their distance is maximal:
\vspace{0.3cm}

\noindent \textbf{Theorem G:} (Theorem \ref{thm:diffuse-comm-sq}) Let $(\sM, \tau)$ be a finite von Neumann algebra, and let $\sP,\sQ\subseteq \sM$ be diffuse orthogonal subalgebras. Then
\[
\dkk(\sP,\sQ)=1=\dmt(\sP,\sQ).
\]

However, in Example \ref{exam:ortho-converse-not-true}, we provide an example showing that the converse need not hold. In contrast, we observe that no two subfactors in the family $(\mathscr{R}_{\mathsf{H}_\alpha})_{\alpha\in[0,\pi)}$ are orthogonal in the sense of Popa. Nevertheless, whenever two such spin model subfactors are maximally distant, their interior angle over their intersection is $\pi/2$, showing that orthogonality in terms of the interior angle is distinct from orthogonality in the sense of Popa; see Remark~\ref{rem:popa-ortho-not-same}.

In the final part of the paper, we turn our attention to distances in group von Neumann algebras of countable discrete groups. In this setting we establish the following: 
\vspace{0.3cm}

\noindent \textbf{Theorem H:} (Theorem \ref{thm:mt_is_trace}) Let $\mF_{2}= \langle a, b \rangle$ and consider a unitary $u \in L(\langle b \rangle)$. Then 
    \[
    \dmt (L(\langle a \rangle), u L(\langle a \rangle)u^*) = \sqrt{1- |\tau(u)|^4}.
    \]    

As a direct consequence (see Corollary \ref{cor:mt-attains-val}), for every $r \in [0,1]$, there exists a unitary $u \in L(\mathbb{F}_2)$ such that
$$\dmt(L(\langle a \rangle),uL(\langle a \rangle)u^*) = r.$$

This continuous spectrum of distances provides a compelling contrast to recent work \cite{Gupta_kumar_2024, kumar-2026}, which shows that the distance between $L(H)$ and $L(K)$ for distinct subgroups $H, K$ is always $1$. Thus, our construction yields maximally distant subalgebras that do not necessarily arise from subgroups. In particular (see Theorem \ref{thm:dkk-1-for-non-tiv-alg}), we show that there exists $u \in L(\langle b \rangle)$ such that $uL(\langle a \rangle)u^*$ is not a group von Neumann algebra associated with any subgroup of $\mathbb{F}_2$, yet

$$\dkk(L(\langle a \rangle),uL(\langle a \rangle)u^*) = 1 = \dmt(L(\langle a \rangle),uL(\langle a \rangle)u^*).$$

A crucial ingredient in our arguments for group von Neumann algebras is a Fourier-analytic property (see Proposition \ref{prop:haar-uni-non-trivial}): for every countable discrete group $G$ with $\vert{}G\vert{} \ge 4$, there exists a unitary $u \in L(G)$ satisfying $\tau(u)=0$ and having at least two nonzero distinct Fourier coefficients. Although this specific fact may be known to experts, we include a rigorous proof for completeness, as it underpins our main algebraic constructions.

The paper is organized as follows. \S\ref{sec:prelims} collects necessary preliminaries on the Kadison-Kastler distance, the Mashood-Taylor distance, and the interior angle. \S\ref{sec:spin-model} focuses on the Spin Model subfactors, explicitly computing their distances in the $2 \times 2$ case and analyzing the continuous spectrum of these geometric invariants. We have shown the existence of uncountably many regular $n\times n$ spin model subfactors in \S\ref{subsec:regularity}. In \S\ref{sec:dist-comm-sq}, we establish our general theorems demonstrating that commuting squares of $\II_1$-factors and diffuse subalgebras force maximal distance. Finally, \S\ref{sec:gr-vNa-dist} shifts to the group von Neumann algebra setting, proving the explicit distance formulas for unitary conjugates in $L(\mathbb{F}_2)$ and containing the requisite Fourier-analytic constructions.

    
    \section{Preliminaries}
    \label{sec:prelims}
    
The purpose of this section is to introduce the notation and conventions that will be used throughout the paper. We write $\mn{n}$ for the algebra of $n\times n$ complex matrices, and denote by $\sU(\mn{n})$ the unitary group of $\mn{n}$ and the set of all $n\times n$ Hadamard unitaries in $\mn{n}$ by $\hn{n}$. The diagonal masa in $\mn{n}$ is denoted by
\[
\Delta^{(n)}:= \left\{
\begin{pmatrix}
\lambda_1 & 0 & \cdots & 0 \\
0 & \lambda_2 & \cdots & 0 \\
\vdots & \vdots & \ddots & \vdots \\
0 & 0 & \cdots & \lambda_n
\end{pmatrix}
: \lambda_1,\ldots,\lambda_n \in \C
\right\}.
\]

For matrices $A \in \mn{m \times n}$ and $B \in \mn{l \times k}$, we define their Kronecker product $A\otimes B \in \mn{ml \times nk}$ by
\[
A\otimes B:= [B_{ij}A]_{1\le i\le l,\,1\le j\le k}.
\]
Notice that, for any $A \in \mn{n}$ and $p \in \N$, we have
\[
A\otimes \mI_p
=
\mathrm{diag}(\underbrace{A, \dots, A}_{p\text{ copies}})\in \mn{pn}.
\]
Note that for two $n\times n$ matrices $A, B\in \mn{n}$, their Kronecker product $A\otimes B$ belongs to $\mn{n^2}$. Moreover, the set of such Kronecker products spans the entire algebra $\mn{n^2}$. Thus, throughout this paper, we use the following definition of the tensor product algebra:
\[
\mn{n} \otimes \mn{n}:= \spn \left\{ A\otimes B: A, B \in \mn{n} \right\}=\mn{n^2}.
\]

Finally, throughout this paper, we use the notation
\[
\sB \subseteq^E \sA \subseteq_e \subseteq \sC
\]
to indicate that \(E:\sA\to\sB\) is a conditional expectation from the \(C^*\)-algebra \(\sA\) onto \(\sB\), that \(\sC\) is the Jones basic construction associated with the inclusion \(\sB\subseteq\sA\) with respect to \(E\), and that \(e\in\sC\) is the corresponding Jones projection.

\begin{definition}[Commuting square]\label{def:comm-sq}
    Let $\sB \subseteq \sC, \sD \subseteq \sA$ be $C^*$-algebras, and let
    \[
        E_{\sB}:\sA\to\sB,\qquad
        E_{\sC}:\sA\to\sC,\qquad
        E_{\sD}:\sA\to\sD
    \]
    be conditional expectations. We say that the quadruple
    $(\sA,\sC,\sD,\sB)$ forms a \emph{commuting square with respect to
    $E_{\sB}$, $E_{\sC}$, and $E_{\sD}$} if
    \[
        E_{\sC}\circ E_{\sD}
        = E_{\sB}
        = E_{\sD}\circ E_{\sC}.
    \]
    We often represent this commuting square by the diagram
    \[
        \begin{array}{ccc}
            \sD & \subset & \sA \\
            \cup & & \cup \\
            \sB & \subset & \sC.
        \end{array}
    \]
\end{definition}

Here, we briefly recall the definition of the interior angle between $\II_1$-subfactors (see \cite{bakshi-etal-2019}).

\begin{definition}{\cite[Definition 2.2]{bakshi-etal-2019}}
  Let $(\sN \subseteq \sP, \sQ \subseteq \sM)$ be a quadruple of $\II_1$-factors. The interior angle between $\sP$ and $\sQ$ is defined by
\[
\operatorname{cos}(\alpha_{\sN}^{\sM})(\sP,\sQ)= \frac{\langle e_{\sP}- e_{\sN}, e_{\sQ}-e_{\sN}\rangle_{\tau}}{\|e_{\sP}- e_{\sN}\|_{\tau}\|e_{\sQ}- e_{\sN}\|_{\tau}},
\]
where $e_{\sP}$, $e_{\sQ}$, and $e_{\sN}$ denote the corresponding Jones projections, and $\tau$ denotes the unique faithful normal trace on $\sM$.
\end{definition}

Let $\sN\subseteq\sM$ be a unital inclusion of $\II_1$-factors. It was shown in \cite[Proposition 2.4]{bakshi-etal-2019} that, for intermediate subfactors $\sN\subseteq\sP,\sQ\subseteq\sM$, the quadruple $(\sN\subseteq\sP,\sQ\subseteq\sM)$ forms a commuting square if and only if
\[
\alpha^\sM_\sN(\sP,\sQ)=\frac{\pi}{2}.
\]

\subsection{Kadison--Kastler Distance}

The perturbation theory of operator algebras on Hilbert spaces was initiated by R.~V.~Kadison and D.~Kastler in 1972 \cite{Kadison_Kastler}. In their work, they introduced a notion of distance between operator algebras acting on a Hilbert space $\sH$. This notion, however, is more generally defined for subspaces of an arbitrary normed space. We begin by recalling the definition of the Kadison--Kastler distance for subspaces of normed spaces and then record some standard facts that will be used in the setting of subalgebras of $\Bh$.

For any normed space $\sX$, we denote its closed unit ball by $\ball{\sX}$. For a subset $\sS$ of $\sX$ and an element $x\in\sX$, the distance from $x$ to $\sS$ is defined by
\[
d(x,\sS):=\inf\{\|x-s\|:s\in\sS\}.
\]

For any two subspaces $\sY$ and $\sZ$ of a normed space $\sX$, the Kadison--Kastler distance between them, denoted by $\dkk(\sY,\sZ)$, is defined as the Hausdorff distance between their closed unit balls.

\begin{definition}[{\cite[Definition A]{Kadison_Kastler}}]
    \[
    \dkk(\sY,\sZ)
    :=\max\left\{
    \sup_{y\in\ball{\sY}}d(y,\ball{\sZ}),
    \sup_{z\in\ball{\sZ}}d(z,\ball{\sY})
    \right\}.
    \]
\end{definition}

\begin{remark}\label{KK-facts}
  Let $\sX$ be a normed space. The following facts about the Kadison--Kastler distance are well known:
  \begin{enumerate}
    \item[(i)] $0\le\dkk(\sY,\sZ)\le1$ for all subspaces $\sY,\sZ$ of $\sX$.
    \item[(ii)] $\dkk$ is a metric on the collection of all subspaces of $\sX$ (see \cite[\S~I.4]{viro-topo}).
  \end{enumerate}
\end{remark}

\noindent For further details, see \cite{Kadison_Kastler} or \cite{Ch_et_al_2010}.

\subsection{Mashood--Taylor distance}
Let $(\sM, \tau)$ be a finite von Neumann algebra. Consider it GNS Hilbert space $L^2(\sM, \tau)$. Throughout we'll denote the canonical embedding of $\sM$ into $L^2(\sM, \tau)$ by 
\[
\eta : \sM \to L^2(\sM, \tau).
\]
The inner product on $L^2(\sM, \tau)$ is given by:
\[
\langle \eta(x), \eta(y)\rangle_{\tau}:=\tau(y^*x) \quad (x, y\in \sM).
\]
Note that $\tau$ induces a bounded linear functional on $\eta(\sM)$ and thus extends to the whole of $L^2(\sM, \tau)$. We'll denote this extension again by $\tau$ and we write $\tau:L^2(\sM, \tau)\to \C$.
For any fixed $a \in \sM$, the map $L_a : L^2(\sM, \tau) \to L^2(\sM, \tau)$ given by $L_a(\eta(x)) = \eta(ax)$ extends to a bounded linear operator on $L^2(\sM, \tau)$ and $a \mapsto L_a$ gives a faithful representation of $\sM$ on $L^2(\sM, \tau)$. Similar conclusion can be done for $b\in \sM$ with $R_b(\eta(x)) = \eta(xb)$.

Let $\sN$ be a von Neumann subalgebra of $\sM$ and $E_\sN$ be the unique $\tau$-preserving conditional expectation from $\sM$ onto $\sN$. Then it can be proved that there is the following isometry:
\[
L^2(\sN, \tau) \cong \overline{\eta(\sN)}^{\|\cdot \|_\tau}\subseteq L^2(\sM, \tau).
\]
Using this we'll treat $L^2(\sN, \tau) \subseteq L^2(\sM, \tau)$. The projection in $\sB(L^2(\sM, \tau))$ with range $L^2(\sN, \tau)$ is called the Jones' projection of the inclusion $\sN \subseteq \sM$ and denoted by $e_\sN$. Furthermore, we have for all $x\in \sM$
\begin{equation}\label{eqn:jones-proj}
    e_\sN(\eta(x)) = \eta(E_\sN(x)).
\end{equation}
For more on this, see \cite{jones-1983, Pimsner_Popa_1986}.

Given a $\II_1$-factor $\sM$ with a faithful normal tracial state $\tau$ and two subfactors $\sP$ and $\sQ$, Mashood and Taylor (see \cite{Mashood-Taylor-1988}) considered the Hausdorff distance between $\ball{\sP}$ and $\ball{\sQ}$ with respect to $\|\cdot\|_{\tau}$. In this setting, Christensen \cite{christensen-79} was the first to study the distance using the trace norm. It is worth noting that the Mashood--Taylor and Christensen distances induce the same topology on the collection of tracial von Neumann subalgebras; see, for instance, \cite{Mashood-Taylor-1988, Gupta_kumar_2024}.

The Mashood--Taylor distance can, in fact, be defined in the more general setting as follows:

\begin{definition}[\cite{Mashood-Taylor-1988}]\label{def:concrete_definiton}
    For any pair $\sP, \sQ $ of subalgebras of $(\sM, \tau)$, the Mashood-Taylor distance between
$\sP$ and $\sQ$ is defined as
\begin{equation}\label{dMT-expresssion}
\dmt(\sP, \sQ) = d_{H, \|\cdot\|_\tau } \Big(\eta(\ball{\sP}), \eta(\ball{\sQ})\Big)
= \max \left \{\sup_{\xi\in \eta(\ball{P})} \dist(\xi, \eta(\ball{Q})), \sup_{\zeta\in \eta(\ball{Q})}\dist(\zeta, \eta(\ball{P}))\right\}
\end{equation}
\end{definition}

\begin{remark}
   Let $(\sM, \tau)$ be a finite von Neumann algebra. The following elementary facts are well-known; for a proof, see, for instance, \cite[Lemma 5.6]{Gupta_kumar_2024}:
\begin{enumerate}
    \item[(i)] $\dmt$ is a semi-metric on the collection of subalgebras of $\sM$.
    \item[(ii)] $\dmt$ is a metric on the collection of von Neumann subalgebras of $\sM$.
\end{enumerate}
\end{remark}

The following observation is well-known and useful for providing an equivalent definition of the Mashood--Taylor distance. We provide a proof for the sake of completeness.

\begin{prop}\label{prop:best_approximation}
 Let $(\sM, \tau)$ be a finite von Neumann algebra and $\sN$ be a von Neumann subalgebra with common unit. Then,
 \[
 \dist \Big(\eta(x), \eta(\ball{\sN})\Big)= \|\eta(x)- \eta(E_{\sN}(x))\|_{\tau} =  \dist \Big(\eta(x), L^{2}(\sN)\Big)
\]
for all $x  \in \ball{\sM}$, where $\dist \Big(\eta(x), \eta(\ball{\sN})\Big)= \inf_{z\in \ball{\sN}}\|\eta(x)-\eta(z)\|_{\tau}$.
\end{prop}
\begin{proof}
Let $x \in \ball{\sM}$. Since $\eta(\ball{\sN})$ is complete (see \cite[Proposition 2.6.4]{Anantharaman_Popa_1986}) and convex in $L^{2}(\sM, \tau)$, there exists
a unique $y_0 \in \ball{\sN}$ such that
\(
\|\eta(x)- \eta(y_0)\|_{\tau}= \dist (\eta(x), \eta(\ball{\sN})).
\)  We assert that   $y_0= E_{\sN}(x).$

Let $e_{\sN}: L^{2}(\sM) \rightarrow L^{2}(\sN):=
\overline{\eta(\sN})^{\|.\|_{\tau}}$ be the Jones projection. Then,
it is known that
\(
\|e_{\sN}(\eta(w))- \eta(w)\|_{\tau}= \dist (\eta(w), L^2(\sN))
\)
for all $w \in \sM$. In particular, for above $x \in \ball{\sM}$,
\[
\|\eta(E_{\sN}(x))- \eta(x)\|_{\tau}= \|e_{\sN}(\eta(x))-
\eta(x)\|_{\tau}= \dist (\eta(x), L^2(\sN)).
\]
Moreover,   $E_{\sN}(x)\in \ball{\sN}$ because   $x \in \ball{\sM}$. So, 
\[
\dist (\eta(x), \eta(\ball{\sN}))= \|\eta(x)- \eta(y_0)\|_{\tau} \leq \|\eta(x)- \eta(E_{\sN}(x))\|_{\tau}= \dist (\eta(x), L^2(\sN))\leq \dist (\eta(x), \eta(\ball{\sN})).
\]
Thus, by the uniqueness of the best approximation element $\eta(y_0)$ of
$\eta(x)$ in $\eta(\ball{\sN})$, we get \( y_0= E_{\sN}(x)\). This proves
our assertion. In particular, we also observe that 
\[
\dist (\eta(x),
\eta(\ball{\sN}))= \|\eta(x)- \eta(E_{\sN}(x))\|_{\tau}=  \dist (\eta(x), L^{2}(\sN)),
\]
as was desired.
\end{proof}

Using Proposition \ref{prop:best_approximation}, we obtain the following equivalent definition of the Mashood--Taylor distance. Throughout this article, we will use only this definition.

\begin{theorem}\label{thm:mashood_taylor}
Let $(\sM, \tau)$ be a finite von Neumann algebra. Then for any von Neumann subalgebras $\sP$ and $\sQ$, we have 
\[
\dmt(\sP, \sQ)= \max \left \{\sup_{x\in \ball{\sP}}\|\eta (x)-\eta(E_{\sQ}(x))\|_{\tau}, \sup_{z\in \ball{\sQ}}\|\eta(z)-\eta(E_{\sP}(z))\|_{\tau}\right\}
\]
where $E_{\sP}$ and $E_{\sQ}$ are the unique trace preserving conditional expectation from $\sM $ onto $\sP$ and $\sQ$ respectively.
\end{theorem}

The following proposition shows the relationship between the Mashood--Taylor distance and the Kadison--Kastler distance.
\begin{prop}\label{prop:mt_and_kk} 
Let $(\sM, \tau)$ be a finite von Neumann algebra. Then for any two subalgebras  $\sP$ and $\sQ$ of $\sM$,
 \[
 \dmt (\sP, \sQ) \leq \dkk(\sP, \sQ).
 \]
\end{prop}
\begin{proof}
    The proof follows from the following well-known fact that for any $x\in \sM$,
    \[
    \|\eta(x)\|_{\tau}\leq \|x\|_{op}.
    \]
\end{proof}

    \section{Distances and Interior Angles of Spin Model Subfactors}
    \label{sec:spin-model}

In this section, we investigate the relative geometric properties and positioning of spin model subfactors constructed from Hadamard unitaries (see \cite{jones-sunder, nicoara2010, BG2025-1}). Spin model subfactors provide a fascinating and rich family of inclusions within the ambient hyperfinite $\mathrm{II}_1$ factor, $\kR$. When working with such an abundance of subfactors, a natural geometric question arises: how can we rigorously quantify the "distance" or relative orientation between two distinct subfactors within this continuous family?

To answer this, we first establish the foundational framework by detailing the canonical Jones towers and their unitarily twisted counterparts, showing exactly how these subfactors sit inside $\kR$. Building on the Bakshi--Guin equivalence (see {\cite[\S4]{BG2025-1}}), we then specialize to the $2\times 2$ case, where the entire family of spin model subfactors can be elegantly parameterized by a single continuous variable $\alpha \in [0, \pi)$. 

The primary purpose of this section is to explicitly compute the Mashood--Taylor distance, $\dmt$, between any two distinct $2\times 2$ spin model subfactors, $\kR_{\sfH_{\alpha}}$ and $\kR_{\sfH_{\beta}}$. Strikingly, we will prove that this distance is governed by a beautifully simple trigonometric identity (see Theorem \ref{thm:mt-distance-2by2}). Finally, we explore the structural consequences of this calculation, directly linking the Mashood--Taylor distance to the interior angle between the subfactors (see Corollary \ref{cor:dmt-is-same-alpha-2by2}), and demonstrating that a maximal distance of $1$ exactly characterizes when the subfactors and their intersection form a commuting square (see Corollary \ref{cor:comm-sq-d-1-2times2}).

We begin with a few lemmas concerning Jones' basic constructions. Let $\tau_{k-1}$ and $\tau_1$ denote the normalized traces on $\mn{n}^{\otimes (k-1)}$ and $\mn{n}$ respectively, so that $\tau_k = \tau_{k-1} \otimes \tau_1$.

\begin{lem}\label{lem:jones-proj-mn-gen}
    Let $\eta_k : \mn{n}^{\otimes k} \to L^2(\mn{n}^{\otimes k}, \tau_k)$ be the canonical inclusion mapping the algebra into its GNS Hilbert space. The Jones projection $e_k$ decomposes as $\operatorname{id}_{L^2(\mn{n}^{\otimes (k-1)})} \otimes e_1$, where $e_1$ is the Jones projection for the base case $\C\cong \C \mI_n \subseteq \mn{n}$.
\end{lem}
\begin{proof}
    By the standard properties of the GNS construction for tensor products of matrix algebras, the Hilbert space factorizes up to unitary equivalence
    $$L^2\left(\mn{n}^{\otimes k}, \tau_k\right) \cong L^2\left(\mn{n}^{\otimes (k-1)}, \tau_{k-1}\right) \otimes L^2(\mn{n}, \tau_1).$$
    
    Under this identification, for any elementary tensor $x \otimes y \in \mn{n}^{\otimes (k-1)} \otimes \mn{n}$, the canonical inclusion factors as
    $$\eta_k(x \otimes y) = \eta_{k-1}(x) \otimes \eta_1(y).$$
    
    The trace-preserving conditional expectation $E_k : \mn{n}^{\otimes (k-1)} \otimes \mn{n} \to \mn{n}^{\otimes (k-1)} \otimes \mathbb{C}$ operates by taking the trace of the second tensor factor, i.e.
    $$E_k(x \otimes y) = x \otimes \tau_1(y)\mI_n.$$
    
    where $\mI_n$ is the identity matrix in $\mn{n}$.
    
    By definition, the Jones projection $e_k$ is the orthogonal projection onto the closed subspace generated by $\eta_k\left(\mn{n}^{\otimes (k-1)}\right)$. Its action on the dense subspace $\eta_k\left(\mn{n}^{\otimes k}\right)$ is characterized by $e_k \eta_k(m) = \eta_k(E_k(m))$ for all $m \in \mn{n}^{\otimes k}$. Applying this to elementary tensors yields
    $$e_k(\eta_{k-1}(x) \otimes \eta_1(y)) = \eta_k(E_k(x \otimes y)) = \eta_k(x \otimes \tau_1(y)\mI_n) = \eta_{k-1}(x) \otimes \eta_1(\tau_1(y)\mI_n).$$
    
    Observe that the action leaves the first tensor factor $\eta_{k-1}(x)$ completely invariant. On the second factor, the map $\eta_1(y) \mapsto \eta_1(\tau_1(y)\mI_n)$ is precisely the action of the Jones projection $e_1$ corresponding to the conditional expectation $E_1 : \mn{n} \to \mathbb{C}$ from the base case. Therefore, as an operator on $L^2\left(\mn{n}^{\otimes (k-1)}, \tau_{k-1}\right) \otimes L^2(\mn{n}, \tau_1)$, we have $e_k = \operatorname{id}_{L^2\left(\mn{n}^{\otimes (k-1)}, \tau_{k-1} \right)} \otimes e_1$. 
\end{proof}

\begin{lem}\label{lem:hyperfinite-ii-1-basic-cons}
    Let $k \ge 1$ be an integer. Consider the unital inclusion of finite-dimensional von Neumann algebras $\mn{n}^{\otimes (k-1)}\subseteq \mn{n}^{\otimes k}$, formed by the canonical identification $\mn{n}^{\otimes (k-1)} \cong \mn{n}^{\otimes (k-1)} \otimes \mathbb{C} \subseteq \mn{n}^{\otimes (k-1)} \otimes \mn{n}$. Let $\tau_k$ be the canonical normalized trace on $\mn{n}^{\otimes k}$, and $E_k: \mn{n}^{\otimes k} \to \mn{n}^{\otimes (k-1)}$ be the unique $\tau_k$-preserving conditional expectation. Then the Jones basic construction $\langle \mn{n}^{\otimes k}, e_k \rangle$ is spatially isomorphic to $\mn{n}^{\otimes (k+1)}$.
\end{lem}
\begin{proof}
Let $\sM:= \mn{n}^{\otimes k}$ and $\sA := \mn{n}^{\otimes (k-1)}$
    The Jones basic construction algebra, denoted $\langle \sM, e_k \rangle$, is the von Neumann algebra acting on $L^2\left(\sM, \tau_k\right)$ generated by $\sM$ (acting by left multiplication) and the Jones projection $e_k$.
    
    From the factorization $\sM = \sA \otimes \mn{n}$ and Lemma \ref{lem:jones-proj-mn-gen}, we can write the generators of the basic construction as
    $$\langle \sM, e_k \rangle = \left( (\sA \otimes \mn{n}) \cup \{ \operatorname{id}_{L^2(\sA)} \otimes e_1 \} \right)''.$$
    
    Because the operators in $\sA$ act strictly on the first Hilbert space factor $L^2(\sA, \tau_{k-1})$, and the operators in $\mn{n}$ along with the projection $e_1$ act strictly on the second Hilbert space factor $L^2(\mn{n}, \tau_1)$. Thus, we have
    $$\langle \sM, e_k \rangle \cong \sA \otimes \langle \mn{n}, e_1 \rangle.$$
   
    Now the basic construction for the inclusion $\mathbb{C} \subseteq \mn{n}$ (see \cite[Theorem 5.1]{ghoshkumar-kk2026}) yields:
    $$\langle \mn{n}, e_1 \rangle \cong \mn{n} \otimes \mn{n}.$$
    
    Hence, we get
    $$\langle \sM, e_k \rangle \cong \mn{n}^{\otimes (k-1)} \otimes (\mn{n} \otimes \mn{n})\cong \mn{n}^{\otimes (k+1)}.$$
    
    This completes the proof.
\end{proof}

\begin{lem}\label{lem:index-invariance}
 For each $k \ge 0$, consider the inclusion
$\mn{n}^{\otimes k} \subseteq^{E_k} \mn{n}^{\otimes (k+1)}$,
where $E_k$ denotes the canonical trace-preserving conditional expectation, characterized by
\(    \tau_k \circ E_k = \tau_{k+1}. 
   \)   
   Then $\operatorname{Ind}(\widetilde{E_k})= \operatorname{Ind}(E_0)$ for each $k$, where $\widetilde{E_k}$ is the Watatani dual coditional expectation corresponding to $E_k$.
\end{lem} 
\begin{proof}
    It follows from \cite[Proposition 2.3.4]{watatani1990}.
\end{proof}

\begin{lem}\label{lem:dual-exp-mn}
For each $k \ge 1$, consider the inclusion
$\mn{n}^{\otimes k} \subseteq^{E_k} \mn{n}^{\otimes (k+1)}$,
where $E_k$ denotes the canonical trace-preserving conditional expectation, characterized by
\(    \tau_{k+1} \circ E_k = \tau_{k+1}.
   \)
Then the Watatani dual conditional expectation $\widetilde{E_k}$ preserves the canonical normalized trace $\tau_{k+2}$. That is, $\tau_{k+1} \circ \tilde{E}_k = \tau_{k+2}$.
\end{lem}
\begin{proof}
    Let $M_k = \mathbb{M}_n^{\otimes k}$. The inclusion $M_k \subseteq M_{k+1}$ is canonically identified as $x \mapsto x \otimes \mathbb{I}_n$.

    The canonical trace-preserving conditional expectation $E_k: \mathbb{M}_n^{\otimes (k+1)} \to \mathbb{M}_n^{\otimes k}$ is given by 
    $$E_k(x \otimes y) = x \otimes \tau_1(y)\mathbb{I}_n$$
   for $x \in \mathbb{M}_n^{\otimes k}$, $y \in \mathbb{M}_n$ and $M_k$ is identified with $M_k \otimes \mathbb{C}\mathbb{I}_n$. 
    
Since elements of the form $x e_{k+1} y$ span $M_{k+2}$ for all $x,y \in M_{k+1}$, it suffices to verify the trace-preserving property on these elements. First, let us evaluate the left-hand side. Applying $\tau_{k+1}$ to the definition of $\tilde{E}_k$ and using Lemma \ref{lem:index-invariance} and trace preserving property of $E_k$, we get
    $$\tau_{k+1}(\tilde{E}_k(x e_{k+1} y)) =  \tau_{k+1}\left((\operatorname{Ind}(E_0))^{-1}E_{k}(xy)\right)=\tau_{k+1}\left(\frac{1}{n^2} x y\right) = \frac{1}{n^2} \tau_{k+1}(xy).$$

    Now, let us evaluate the right-hand side, $\tau_{k+2}(x e_{k+1} y)$. Using the trace property of $\tau_{k+2}$, we get $$\tau_{k+2}(x e_{k+1} y) = \tau_{k+2}(yx e_{k+1}).$$

    We now evaluate the trace of $w e_{k+1}$ for any $w \in M_{k+1}$. The $M_{k+2}$ can be decomposes as $\mathbb{M}_n^{\otimes k} \otimes (\mathbb{M}_n \otimes \mathbb{M}_n)$, and by the Lemma \ref{lem:jones-proj-mn-gen} the projection is $e_{k+1} = \mathbb{I}_{n^k} \otimes e_1$. A general element $w \in M_{k+1}$ takes the form $a \otimes (b \otimes \mathbb{I}_n)$.

    Thus, $w e_{k+1}$ decomposes over the tensor product as $a \otimes ((b \otimes \mathbb{I}_n)e_1)$. Evaluating the trace yields:
    $$\tau_{k+2}(w e_{k+1}) = \tau_k(a) \cdot \tau_2((b \otimes \mathbb{I}_n)e_1).$$

    Using the trace property in $M_2$ and the fact that $e_1$ is a projection ($e_1^2 = e_1$), we get
    $$\tau_2((b \otimes \mathbb{I}_n)e_1) = \tau_2(e_1 (b \otimes \mathbb{I}_n) e_1).$$
    
    By the defining property of the Jones projection, $e_1 (b \otimes \mathbb{I}_n) e_1 = E_1(b) e_1 = \tau_1(b)e_1$. Therefore
    $$\tau_2(e_1 (b \otimes \mathbb{I}_n) e_1) = \tau_2(\tau_1(b)e_1) = \tau_1(b)\tau_2(e_1).$$

    Because $e_1$ is a rank-1 projection in $\mathbb{M}_n \otimes \mathbb{M}_n \cong \mathbb{M}_{n^2}$, its normalized trace is $\frac{1}{n^2}$. Substituting this back gives:
    $$\tau_{k+2}(w e_{k+1}) = \tau_k(a) \tau_1(b) \frac{1}{n^2} = \tau_{k+1}(w) \frac{1}{n^2}.$$

    Letting $w = yx$, we obtain,
    $$\tau_{k+2}(x e_{k+1} y) = \frac{1}{n^2} \tau_{k+1}(yx)= \tau_{k+1}(xy).$$
    Thus, we get
    $$\tau_{k+1}(\tilde{E}_k(x e_{k+1} y)) = \tau_{k+2}(x e_{k+1} y).$$
    This completes the proof!
\end{proof}

\begin{prop}\label{prop:jones-tower-for-hyperfin-ii-1}
Consider the unital inclusion $\C\mI_n \subseteq^{E_0} \mn{n}$, where $E_0 = \frac{1}{n}\operatorname{Tr} \otimes \mI_n.$ The infinite Jones tower obtained by iteratively applying the Jones basic construction is given by the sequence
\[
\C \subseteq^{E_0}_{e_1} \mn{n}
\subseteq^{E_1}_{e_2} \mn{n}^{\otimes 2}
\subseteq \cdots
\subseteq \mn{n}^{\otimes k}
\subseteq^{E_k}_{e_{k+1}} \mn{n}^{\otimes (k+1)}
\subseteq \cdots,
\]
where, for each $k \geq 0$, $E_k$ denotes the canonical trace-preserving conditional expectation and the corresponding Jones projection is given by $e_{k+1}
=
\operatorname{id}_{L^2(\mn{n}^{\otimes k})} \otimes e_1.$ Moreover the Watatani's dual conditional expectation, $\widetilde{E_k}= E_{k+1}$.
\end{prop}
\begin{proof}
    Follows from Lemma \ref{lem:jones-proj-mn-gen}, Lemma \ref{lem:hyperfinite-ii-1-basic-cons} and Lemma \ref{lem:dual-exp-mn}.
\end{proof}

\begin{prop}[{\cite[Proposition 2.2]{nicoara2010}}]\label{prop:nic-tower}
Consider the unital inclusion $\Deltan\subseteq^{E_{\Deltan}} \mn{n}$, where $E_{\Deltan}$ is the canonical $\operatorname{Tr}$ preserving conditional expectation. The infinite Jones tower obtained by iteratively applying the Jones basic construction is given by the sequence
    \[
    \Deltan\subseteq_{e_2} \mn{n}\subseteq_{e_3}
    \cdots
    \subseteq
    \mn{n}^{\otimes (k-1)} \otimes \Deltan
    \subseteq_{e_{2k}}
    \mn{n}^{\otimes k}
    \subseteq_{e_{2k+1}}
    \mn{n}^{\otimes k} \otimes \Deltan
    \subseteq \cdots
    \]
    Where the corresponding Jones projection $e_j$ is given recursively by
    \begin{align*}
        e_2 &:= \frac{1}{n} \sum_{i,j=1}^n E_{ij}, \quad
        e_3 := \sum_{i=1}^n (E_{ii} \otimes E_{ii}), \\
        e_{2k} &:= \bigotimes_{1}^{k-1} (\mI_n \otimes e_2), \quad
        e_{2k+1} := \bigotimes_{1}^{k-1} (\mI_n \otimes e_3),
        \qquad k=2,3,\dots.
    \end{align*}
\end{prop}

Before introducing the spin model subfactors we briefly recall the \emph{Hadamard unitaries}.

A complex \emph{Hadamard matrix} is a matrix $H\in \mn{n}$ whose entries all have absolute value $1$ and whose rows are mutually orthogonal. A general classification of complex Hadamard matrices is not currently available. A catalogue of most of the known complex Hadamard matrices can be found in \cite{woj-karol-2005}. Complete classifications are known for $n\leq 5$ \cite{haagerup1995} and for self-adjoint Hadamard matrices of order $6$ (\cite{Kyle-Nicoara-2008}).

We say that a matrix $W\in \mn{n}$ is a \emph{Hadamard unitary} if $W=\sqrt{n}H$ for some Hadamard matrix $H\in \mn{n}$. It is straightforward to verify that $u\in \mn{n}$ is a Hadamard unitary if and only if $u$ is unitary and
\[
|u_{ij}|=\frac{1}{\sqrt{n}},\qquad i,j=1,2,\dots,n,
\]
see \cite[Proposition 3.10]{ghoshkumar-kk2026}. Throughout this article, we denote by $\hn{n}$ the set of all Hadamard unitaries in $\mn{n}$, where $n\geq 2$.

\begin{prop}[{\cite[Proposition 2.3]{nicoara2010}}]\label{prop:nic-unitary-tower}
    Let $u$ be an $n\times n$ Hadamard unitary in $\mn{n}$. Consider the unital inclusion $u\Deltan u^* \subseteq^{E_u} \mn{n}$, where $E_u:= \operatorname{Ad}_u \circ E_{\Deltan} \circ \operatorname{Ad}_{u^*}$. The infinite Jones tower obtained by iteratively applying the Jones basic construction is given by the sequence
    \[
    u\Deltan u^*
    \subseteq_{e_2}u_1 \mn{n}u_1^*\subseteq \cdots
    \subseteq
    u_{2k-1}\mn{n}^{\otimes k}(u_{2k-1})^*
    \subseteq_{e_{2k+1}}
    u_{2k}(\mn{n}^{\otimes k}\otimes\Deltan)(u_{2k})^*
    \subseteq_{e_{2k+2}} \cdots.
    \]
    Here, $u_{2k-1}\in \mn{n}^{\otimes k}\otimes\Deltan$ and
    $u_{2k}\in\mn{n}^{\otimes(k+1)}$ are unitaries defined recursively by
    \begin{align*}
        D_u := \sqrt{n}\sum_{i,j=1}^n
        \overline{u_{ij}}(E_{jj}\otimes E_{ii}), &\quad
        u_1 := (u\otimes\mI_n)D_u,\\
        u_{2k} := u_{2k-1}
        \left(\otimes^k\mI_n\otimes u\right),
        &\quad
        u_{2k+1} := \prod_{i=0}^k
        \left(\otimes^i\mI_n\otimes u_1\otimes^{k-i}\mI_n\right),
        \quad k=1,2,\dots.
    \end{align*}
    The Jones projections $e_k$ are the same as those defined in the preceding proposition.
\end{prop}

\begin{definition}[Spin Model Subfactor]\label{def:spin-model-subfactor}
    Let $u \in \mn{n}$ be a Hadamard unitary. The \emph{spin model subfactor}, denoted by $\kR_u$, is defined as the hyperfinite $\mathrm{II}_1$ factor obtained by taking the weak closure of the inductive limit of the unitarily twisted Jones tower constructed in Proposition \ref{prop:nic-unitary-tower} with respect to the unique canonical normalized trace. 
    
    Similarly, the ambient hyperfinite $\mathrm{II}_1$ factor, denoted by $\kR$, is defined as the weak closure of the inductive limit of the standard, untwisted Jones tower for the inclusion $\Deltan \subseteq \mn{n}$ given in Proposition \ref{prop:nic-tower}. By construction, $\kR_u \subseteq \kR$.
\end{definition}

\begin{prop}
    Let $M_{k-1} \subseteq M_k \subseteq M_{k+1} = \langle M_k, e_{k+1} \rangle$ be any two consecutive inclusions in the unitarily twisted Jones tower defined in Proposition \ref{prop:nic-unitary-tower}, where $[M_k : M_{k-1}] = n$. Let $\tau$ denote the canonical faithful normalized trace on the ambient matrix algebras. Then Watatani's dual conditional expectation $\tilde{E}_k : M_{k+1} \to M_k$ preserves the trace $\tau$, and consequently, $\tilde{E}_k$ coincides exactly with the unique canonical $\tau$-preserving conditional expectation $E_k : M_{k+1} \to M_k$.
\end{prop}
\begin{proof}
    By the definition of the Jones basic construction, the algebra $M_{k+1}$ is linearly spanned by elements of the form $x e_{k+1} y$, where $x, y \in M_k$.

    Watatani's dual conditional expectation $\tilde{E}_k : M_{k+1} \to M_k$ is uniquely defined on these spanning elements by
    $$\tilde{E}_k(x e_{k+1} y) = \frac{1}{[M_k : M_{k-1}]} x y = \frac{1}{n} x y$$

    According to Proposition \ref{prop:nic-unitary-tower}, the Jones projections $e_{k+1}$ in this unitarily twisted tower are identically equal to the Jones projections in the standard tower (Proposition \ref{prop:nic-tower}).

    Because the canonical normalized trace $\tau$ on the ambient matrix algebra is unitarily invariant (i.e., $\tau(u X u^*) = \tau(X)$), the spatial rotations defining $M_k$ do not alter the trace values. Furthermore, since the projections $e_{k+1}$ are exactly the standard ones, they maintain the standard Markov property with respect to $\tau$ and the algebra $M_k$.

    Specifically, for any $x, y \in M_k$, the trace satisfies the Markov relation with parameter $\frac{1}{n}$:$$\tau(x e_{k+1} y) = \frac{1}{n} \tau(xy)$$

    We now verify that Watatani's dual expectation $\tilde{E}_k$ preserves $\tau$. Applying the trace to the output of $\tilde{E}_k$ for a generic spanning element and using Lemma \ref{lem:index-invariance}, we get
    $$\tau(\tilde{E}_k(x e_{k+1} y)) = \tau\left( \frac{1}{n} x y \right) = \frac{1}{n} \tau(xy).$$
    Comparing this with the Markov property established above, we get
    $$\tau(\tilde{E}_k(x e_{k+1} y)) = \tau(x e_{k+1} y)$$By linearity, this equality extends to finite sums of such elements. Since these elements span $M_{k+1}$, we conclude that for all $z \in M_{k+1}$:$$\tau(\tilde{E}_k(z)) = \tau(z)$$Thus, Watatani's dual conditional expectation is trace-preserving.
\end{proof}

The structural relationship between the spin model subfactors corresponding to different Hadamard unitaries can be seamlessly visualized. Let $u, v \in \hn{n}$ be two Hadamard unitaries. The respective twisted Jones towers for $u$ and $v$ each embed into the standard Jones tower for the ambient factor. The inclusions between these towers at every step of the basic construction yield the following large commutative diagram, demonstrating how $\kR_u$ and $\kR_v$ sit inside the ambient factor $\kR$:

$$
\begin{array}{ccccc} 
\kR_u & \lhook\joinrel\longrightarrow & \kR & \longleftarrow\joinrel\rhook & \kR_v \\
\vdots & & \vdots & & \vdots \\ \rotatebox{90}{$\lhook\joinrel\longrightarrow$} & & \rotatebox{90}{$\lhook\joinrel\longrightarrow$} & & \rotatebox{90}{$\lhook\joinrel\longrightarrow$} \\ u_{2k}(\mn{n}^{\otimes k}\otimes\Deltan)(u_{2k})^* & \lhook\joinrel\longrightarrow & \mn{n}^{\otimes (k+1)} & \longleftarrow\joinrel\rhook & v_{2k}(\mn{n}^{\otimes k}\otimes\Deltan)(v_{2k})^* \\ \rotatebox{90}{$\lhook\joinrel\longrightarrow$} & & \rotatebox{90}{$\lhook\joinrel\longrightarrow$} & & \rotatebox{90}{$\lhook\joinrel\longrightarrow$} \\ u_{2k-1}\mn{n}^{\otimes k}(u_{2k-1})^* & \lhook\joinrel\longrightarrow & \mn{n}^{\otimes k} \otimes \Deltan & \longleftarrow\joinrel\rhook & v_{2k-1}\mn{n}^{\otimes k}(v_{2k-1})^* \\ \rotatebox{90}{$\lhook\joinrel\longrightarrow$} & & \rotatebox{90}{$\lhook\joinrel\longrightarrow$} & & \rotatebox{90}{$\lhook\joinrel\longrightarrow$} \\ \vdots & & \vdots & & \vdots \\ \rotatebox{90}{$\lhook\joinrel\longrightarrow$} & & \rotatebox{90}{$\lhook\joinrel\longrightarrow$} & & \rotatebox{90}{$\lhook\joinrel\longrightarrow$} \\ u_2(\mn{n}\otimes\Deltan)u_2^* & \lhook\joinrel\longrightarrow & \mn{n}^{\otimes 2} & \longleftarrow\joinrel\rhook & v_2(\mn{n}\otimes\Deltan)v_2^* \\ \rotatebox{90}{$\lhook\joinrel\longrightarrow$} & & \rotatebox{90}{$\lhook\joinrel\longrightarrow$} & & \rotatebox{90}{$\lhook\joinrel\longrightarrow$} \\ u_1\mn{n}u_1^* & \lhook\joinrel\longrightarrow & \mn{n}\otimes\Deltan & \longleftarrow\joinrel\rhook & v_1\mn{n}v_1^* \\ \rotatebox{90}{$\lhook\joinrel\longrightarrow$} & & \rotatebox{90}{$\lhook\joinrel\longrightarrow$} & & \rotatebox{90}{$\lhook\joinrel\longrightarrow$} \\ u\Deltan u^* & \lhook\joinrel\longrightarrow & \mn{n} & \longleftarrow\joinrel\rhook & v\Deltan v^* \\ \rotatebox{90}{$\lhook\joinrel\longrightarrow$} & & \rotatebox{90}{$\lhook\joinrel\longrightarrow$} & & \rotatebox{90}{$\lhook\joinrel\longrightarrow$} \\ \C & \lhook\joinrel\longrightarrow & \Deltan & \longleftarrow\joinrel\rhook & \C 
\end{array}
$$

\begin{definition}[Hadamard Equivalence]
    Two Hadamard unitaries $u,v\in\hn{n}$ are said to be \emph{Hadamard equivalent}, denoted by $u\hada v$, if there exist diagonal unitaries $D_1,D_2\in\mn{n}$ and permutation matrices $P_1,P_2\in\mn{n}$ such that
    \[
        v=D_1P_1uP_2D_2.
    \]
\end{definition}

\begin{definition}[Bakshi--Guin Equivalence, {\cite[\S~4]{BG2025-1}}]\label{def:bakshi-guin-eq-relation}
    Two Hadamard unitaries $u,v\in\hn{n}$ are said to be \emph{equivalent in the sense of Bakshi and Guin}, denoted by $u\bg v$, if there exist a permutation matrix $P$ and a diagonal unitary matrix $D$ such that
    \[
        v=uPD.
    \]
\end{definition}

For $\alpha\in\R$, consider the $2\times2$ Hadamard unitary
\[
\sfH_\alpha
:=
\begin{bmatrix}
1 & 0\\
0 & e^{i\alpha}
\end{bmatrix}
\fourier{2}
=
\frac{1}{\sqrt{2}}
\begin{bmatrix}
1 & 1\\
e^{i\alpha} & -e^{i\alpha}
\end{bmatrix}
\in\hn{2}.
\]

It is straightforward to verify that
\[
\sfH_\alpha \bg \sfH_\beta
\quad\Longleftrightarrow\quad
\beta\equiv\alpha\pmod{\pi}.
\]
Moreover, it follows from \cite[\S5]{BG2025-1} that, for every $\sfH\in\hn{2}$, there exists an $\alpha\in[0,\pi)$ such that $\sfH\bg\sfH_\alpha$. Consequently,
\[
\hn{2}/\bg\,
=
\Big\{
[\sfH_\alpha]_{\bg}:\alpha\in[0,\pi)
\Big\}.
\]

By \cite[Theorem 4.7 (i)]{BG2025-1}, the spin model subfactors $\kR_{\sfH_\alpha}$, $\alpha\in[0,\pi)$, are pairwise distinct. Furthermore, (by \cite[Theorem 5.8]{BG2025-1}) for distinct $\alpha,\beta\in[0,\pi)$, their intersection is a $\II_1$ subfactor of $\kR$ satisfying
\[
[\kR:\kR_{\sfH_\alpha}\cap\kR_{\sfH_\beta}]=4.
\]

Thus, in the $2\times2$ case, every pair of distinct spin model subfactors of $\kR$ is of the form
\[
\bigl(\kR_{\sfH_\alpha},\kR_{\sfH_\beta}\bigr),
\]
for distinct $\alpha,\beta\in[0,\pi)$.

\begin{theorem}[{\cite[Theorem 5.16]{BG2025-1}}]\label{thm:bakshi-guin}
Let $\alpha,\beta\in[0,\pi)$ be distinct. Then there exists an involutive
outer automorphism $\Theta_{\alpha,\beta}\in\operatorname{Out}(\kR)$, that is,
\[
(\Theta_{\alpha,\beta})^2=\operatorname{id}_{\kR},
\]
such that,
\begin{align*}
\kR_{\sfH_\alpha}\cap\kR_{\sfH_\beta}
&\cong
\left\{
\begin{pmatrix}
x & 0\\
0 & \Theta_{\alpha,\beta}(x)
\end{pmatrix}
:x\in\kR
\right\},\\[1ex]
\kR_{\sfH_\beta}
&\cong
\left\{
\begin{pmatrix}
x & y\\
\Theta_{\alpha,\beta}(y) & \Theta_{\alpha,\beta}(x)
\end{pmatrix}
:x,y\in\kR
\right\},\\[1ex]
\kR_{\sfH_\alpha}
&\cong
\begin{pmatrix}
1&0\\
0&e^{i(\alpha-\beta)}
\end{pmatrix}
\left\{
\begin{pmatrix}
x & y\\
\Theta_{\alpha,\beta}(y) & \Theta_{\alpha,\beta}(x)
\end{pmatrix}
:x,y\in\kR
\right\}
\begin{pmatrix}
1&0\\
0&e^{-i(\alpha-\beta)}
\end{pmatrix}.
\end{align*}
Here, all the algebras on the right-hand side are regarded as subalgebras
of $\mn{2}[\kR]$.
\end{theorem}

The following elementary lemma is easy to prove. We include a proof for completeness.

\begin{lem}\label{lem:trace-preserving}
  Let  $\sM$ be a $\II_1$-factor with the  unique faithful normal trace $\tau_{\sM}$ and let  $\Theta\in \operatorname{Aut}(\sM)$ such that $\Theta^2= \operatorname{id}$. Then, $\Theta$ is trace preserving and for any $x,y \in \sM$ 
  $$\tau_{\sM}(\Theta(x)y) = \tau_{\sM}(x\Theta(y)).$$ 
\end{lem}
\begin{proof}
    Let $\Theta \in \operatorname{Aut}(\sM)$. We can define a new functional on $\sM$ by composing the trace with this automorphism
    $$\tau'(x) := \tau_{\sM}(\Theta(x)).$$
    Clearly this is faithful, normal  functional. It is a normalized trace because $$\tau'(1) = \tau_{\sM}(\Theta(1)) = \tau_{\sM}(1) = 1.$$
    Also it satisfies the trace property: $\tau'(xy) = \tau_{\sM}(\Theta(xy)) = \tau_{\sM}(\Theta(x)\Theta(y)) = \tau_{\sM}(\Theta(y)\Theta(x)) = \tau_{\sM}(\Theta(yx)) = \tau'(yx)$. Because the normalized trace on a II$_1$ factor is unique, it implies that $\tau' = \tau_{\sM}$. Therefore, $\tau_{\sM}(\Theta(x)) =\tau'(x)= \tau_{\sM}(x)$ for all $x \in \sM$. This implies that $\Theta$ is trace preserving. Let $x,y \in \sM$. Then, we get
  \[
  \tau_{\sM}(\Theta(x)y) = \tau_{\sM}(\Theta(\Theta(x)y))= \tau_{\sM}(\Theta(\Theta(x))\Theta(y))= \tau_{\sM}(x\Theta(y)).
  \]
\end{proof}

\begin{lem}\label{lem:condit-exp}
Let $\alpha,\beta\in[0,\pi)$ be distinct and $\Theta_{\alpha, \beta}\in \operatorname{Out}(\kR)$ as given by Theroem \ref{thm:bakshi-guin}. Then the unique trace-preserving conditional expectations $E_{\kR_{\sfH_\alpha}}: \mathbb{M}_2(\kR) \to \kR_{\sfH_{\alpha}}$ and $E_{\kR_{\sfH_\beta}}: \mathbb{M}_2(\kR) \to \kR_{\sfH_\beta}$ are given as follows:

$$E_{\kR_{\sfH_\alpha}}\left( \begin{bmatrix} m_{11} & m_{12} \\ m_{21} & m_{22} \end{bmatrix} \right) = \begin{bmatrix} x & y e^{i(\beta-\alpha)} \\ \Theta_{\alpha, \beta} (y) e^{i(\alpha-\beta)} & \Theta_{\alpha, \beta} (x) \end{bmatrix}$$
where $x = \frac{1}{2}(m_{11} + \Theta_{\alpha, \beta} (m_{22}))$ and $y = \frac{1}{2}(e^{i(\alpha-\beta)}m_{12} + e^{i(\beta-\alpha)}\Theta_{\alpha, \beta} (m_{21}))$, and
$$E_{\kR_{\sfH_\beta}}\left( \begin{bmatrix} m_{11} & m_{12} \\ m_{21} & m_{22} \end{bmatrix} \right) = \begin{bmatrix} \frac{m_{11} + \Theta_{\alpha, \beta} (m_{22})}{2} & \frac{m_{12} + \Theta_{\alpha, \beta} (m_{21})}{2} \\ \Theta_{\alpha, \beta} \left(\frac{m_{12} + \Theta_{\alpha, \beta} (m_{21})}{2}\right) & \Theta_{\alpha, \beta} \left(\frac{m_{11} + \Theta_{\alpha, \beta} (m_{22})}{2}\right) \end{bmatrix}$$
where $m=\begin{bmatrix} m_{11} & m_{12} \\ m_{21} & m_{22} \end{bmatrix} \in \mathbb{M}_2(\kR)$.
 \end{lem}
\begin{proof}
Let 
 $m=\begin{bmatrix} m_{11} & m_{12} \\ m_{21} & m_{22} \end{bmatrix} \in \mathbb{M}_2(\kR)$ be given. The unique normalized trace on $\mathbb{M}_2(\kR)$ is given by $\tau_{\mathbb{M}_2(\kR)}(m) = \frac{1}{2}(\tau_{\kR}(m_{11}) + \tau_{\kR}(m_{22}))$ where $\tau_{\kR}$ is the unique trace on $\kR$. Since $\kR$ is a $\II_1$-factor and $\Theta_{\alpha, \beta}  \in \operatorname{Aut}(\kR)$ is an automorphism such that $\Theta_{\alpha, \beta} ^2= \operatorname{id}$, then by Lemma \ref{lem:trace-preserving}, $\Theta_{\alpha, \beta} $ is trace-preserving and satisfies $\tau_{\kR}(\Theta_{\alpha, \beta} (a)b) = \tau_{\kR}(a\Theta_{\alpha, \beta} (b))$ for all $a, b \in \kR$.

 Let $r= \begin{bmatrix} a & b \\ \Theta_{\alpha, \beta} (b) & \Theta_{\alpha, \beta} (a) \end{bmatrix} \in \kR_{\sfH_\beta}$ be an arbitrary element. Then we get 
 $$\tau_{\mathbb{M}_2(\kR)}((m - E_{\kR_{\sfH_\beta}}(m))r) = 0.$$ 
 Equating $\tau_{\mathbb{M}_2(\kR)}(mr) = \tau_{\mathbb{M}_2(\kR)}(E_{\kR_{\sfH_\beta}}(m)r)$ gives
 $$\frac{1}{2}\tau_{\kR}(m_{11}a + m_{12}\Theta_{\alpha, \beta} (b) + m_{21}b + m_{22}\Theta_{\alpha, \beta} (a)) = \frac{1}{2}\tau_{\kR}(xa + y\Theta_{\alpha, \beta} (b) + \Theta_{\alpha, \beta} (y)b + \Theta_{\alpha, \beta} (x)\Theta_{\alpha, \beta} (a)).$$
 Using the property $\tau_{\kR}(\Theta_{\alpha, \beta} (u)v) = \tau_\kR(u\Theta_{\alpha, \beta} (v))$ and $\Theta_{\alpha, \beta} ^2 = \operatorname{id}$, the right side simplifies to $\tau_\kR(xa + y\Theta_{\alpha, \beta} (b))$. The left side rearranges to $\frac{1}{2}\tau_\kR((m_{11} + \Theta_{\alpha, \beta} (m_{22}))a + (m_{12} + \Theta_{\alpha, \beta} (m_{21}))\Theta_{\alpha, \beta} (b))$.
 Since this must hold for all $a, b \in \kR$, we can equate the terms: $x = \frac{1}{2}(m_{11} + \Theta_{\alpha, \beta} (m_{22}))$, $y = \frac{1}{2}(m_{12} + \Theta_{\alpha, \beta} (m_{21}))$. Therefore, the expectation is 
 $$E_{\kR_{\sfH_\beta}}\left( \begin{bmatrix} m_{11} & m_{12} \\ m_{21} & m_{22} \end{bmatrix} \right) = \begin{bmatrix} \frac{m_{11} + \Theta_{\alpha, \beta} (m_{22})}{2} & \frac{m_{12} + \Theta_{\alpha, \beta} (m_{21})}{2} \\ \Theta_{\alpha, \beta} \left(\frac{m_{12} + \Theta_{\alpha, \beta} (m_{21})}{2}\right) & \Theta_{\alpha, \beta} \left(\frac{m_{11} + \Theta_{\alpha, \beta} (m_{22})}{2}\right) \end{bmatrix}.$$ 
 Similarly, take an arbitrary element in $\kR_{\sfH_\alpha}$ and apply the exact same trace-orthogonality procedure. The diagonal terms behave identically to $\kR_{\sfH_\beta}$, yielding the exact same expression for $x$. For the off-diagonal terms, factoring in the phase shift $e^{i(\beta-\alpha)}$ and its conjugate $e^{i(\alpha-\beta)}$, the trace equivalence demands that
 $$y = \frac{1}{2} \left( e^{i(\alpha-\beta)}m_{12} + e^{i(\beta-\alpha)}\Theta_{\alpha, \beta} (m_{21}) \right).$$
 Thus, the conditional expectation onto $\kR_{\sfH_\alpha}$ is
 $$E_{\kR_{\sfH_\alpha}}\left( \begin{bmatrix} m_{11} & m_{12} \\ m_{21} & m_{22} \end{bmatrix} \right) = \begin{bmatrix} x & y e^{i(\beta-\alpha)} \\ \Theta_{\alpha, \beta} (y) e^{i(\alpha-\beta)} & \Theta_{\alpha, \beta} (x) \end{bmatrix}$$
 where the parameters are 
 $$x = \frac{1}{2}(m_{11} + \Theta_{\alpha, \beta} (m_{22}))$$
 $$y = \frac{1}{2}(e^{i(\alpha-\beta)}m_{12} + e^{i(\beta-\alpha)}\Theta_{\alpha, \beta} (m_{21})).$$
\end{proof}

\begin{lem}\label{lem:one-sided-spin}
Let $\alpha,\beta\in[0,\pi)$ be distinct and $\Theta_{\alpha, \beta}\in \operatorname{Out}(\kR)$ as given by Theorem \ref{thm:bakshi-guin}. Let $\tau$ denote the unique normalized trace on $\mathbb{M}_{2}(\kR)$. Then
$$\sup_{m\in \ball{\kR_{\sfH_{\alpha}}}} \Vert{}\eta(m) - \eta(E_{\kR_{\sfH_{\beta}}}(m))\Vert{}_\tau = \vert{}\sin(\beta-\alpha)\vert{}.$$\end{lem}
\begin{proof}
Let $m$ be an arbitrary element in $\ball{\kR_{\sfH_\alpha}}$. From the definition of $\kR_{\sfH_\alpha}$, it has the form, 
$$m = \begin{bmatrix} x & y e^{i(\beta-\alpha)} \\ \Theta_{\alpha, \beta} (y) e^{i(\alpha-\beta)} & \Theta_{\alpha, \beta} (x) \end{bmatrix} \quad \text{for } x, y \in \kR.$$
Using the conditional expectation formula onto $\kR_{\sfH_\beta}$ from the Lemma \ref{lem:condit-exp}, we substitute the entries of $m$ : $m_{11} = x$, $m_{12} = y e^{i(\beta-\alpha)}$, $m_{21} = \Theta_{\alpha, \beta} (y) e^{i(\alpha-\beta)}$, $m_{22} = \Theta_{\alpha, \beta} (x)$. The diagonal entries of $E_{\kR_{\sfH_\beta}}(m)$ are $\frac{1}{2}(x + \Theta_{\alpha, \beta} (\Theta_{\alpha, \beta} (x))) = x$ (since $\Theta_{\alpha, \beta} ^2 = \text{id}$). The $(1,2)$ entry is
$$\frac{1}{2}(m_{12} + \Theta_{\alpha, \beta} (m_{21})) = \frac{1}{2}\left(y e^{i(\beta-\alpha)} + y e^{i(\alpha-\beta)}\right) = y \cos(\beta-\alpha).$$
Thus, the projection is
$$E_{\kR_{\sfH_{\beta}}}(m) = \begin{bmatrix} x & y \cos(\beta-\alpha) \\ \Theta_{\alpha, \beta} (y) \cos(\beta-\alpha) & \Theta_{\alpha, \beta} (x) \end{bmatrix}.$$

Subtracting the projection $E_{\kR_{\sfH_{\beta}}}(m)$ from $m$, we get

\begin{align*}
    z &= \begin{bmatrix} 0 & y(e^{i(\beta-\alpha)} - \cos(\beta-\alpha)) \\ \Theta_{\alpha, \beta} (y)(e^{i(\alpha-\beta)} - \cos(\beta-\alpha)) & 0 \end{bmatrix}\\
&= \begin{bmatrix} 0 & i y \sin(\beta-\alpha) \\ -i \Theta_{\alpha, \beta} (y) \sin(\beta-\alpha) & 0 \end{bmatrix}.
\end{align*}

First, note that
$$z^* z = \sin^2(\beta-\alpha) \begin{bmatrix} 0 & i \Theta_{\alpha, \beta} (y^*) \\ -i y^* & 0 \end{bmatrix} \begin{bmatrix} 0 & i y \\ -i \Theta_{\alpha, \beta} (y) & 0 \end{bmatrix} = \sin^2(\beta-\alpha) \begin{bmatrix} \Theta_{\alpha, \beta} (y^* y) & 0 \\ 0 & y^* y \end{bmatrix}.$$
We now apply the trace on $\mn{2} \otimes \kR$, we get
$$\tau(z^* z) = \frac{1}{2} \left[ \tau_\kR(\sin^2(\beta-\alpha) \Theta_{\alpha, \beta} (y^* y)) + \tau_\kR(\sin^2(\beta-\alpha) y^* y) \right].$$
Because $\Theta_{\alpha, \beta} $ is trace-preserving, $\tau_\kR(\Theta_{\alpha, \beta} (y^* y)) = \tau_\kR(y^* y) = \Vert{}y\Vert{}_{\tau_\kR}^2$. Then, we get 
$$\Vert{}\eta(z)\Vert{}_\tau^2 = \sin^2(\beta-\alpha) \Vert{}y\Vert{}_{\tau_\kR}^2.$$
Therefore, the $\tau$-norm distance for a given $m$ is
$$\Vert{}\eta(m) - \eta(E_{\kR_{\sfH_\beta}}(m))\Vert{}_\tau = \vert{}\sin(\beta-\alpha)\vert{} \Vert{}y\Vert{}_{\tau_\kR}$$
Since $m$ is in the closed unit ball with respect to the operator norm, this implies $m^* m \le 1_{\mn{2} \otimes \kR}$. Computing the $(2,2)$ diagonal entry of $m^* m$
$$(m^* m)_{22} = (y e^{i(\beta-\alpha)})^*(y e^{i(\beta-\alpha)}) + \Theta_{\alpha, \beta} (x)^*\Theta_{\alpha, \beta} (x) = y^* y + \Theta_{\alpha, \beta} (x^* x) \le 1_\kR$$
Applying the trace $\tau_\kR$ to this inequality gives 
$$\tau_\kR(y^* y) + \tau_\kR(\Theta_{\alpha, \beta} (x^* x)) \le 1 \implies \Vert{}y\Vert{}_{\tau_\kR}^2 + \Vert{}x\Vert{}_{\tau_\kR}^2 \le 1.$$
This forces $\Vert{}y\Vert{}_{\tau_\kR} \le \sqrt{1 - \Vert{}x\Vert{}_{\tau_\kR}^2} \le 1$. Thus, we get
$$\Vert{}\eta(m) - \eta(E_{\kR_{\sfH_{\beta}}}(m))\Vert{}_\tau = \vert{}\sin(\beta-\alpha)\vert{} \Vert{}y\Vert{}_{\tau_\kR}\leq \vert{}\sin(\beta-\alpha)\vert{}.$$ 
We can construct a specific element $m_0 \in \ball{\kR_{\sfH_{\alpha}}}$ that attains this bound. Choose $x = 0$ and $y = 1_\kR$. This gives the matrix
$$m_0 = \begin{bmatrix} 0 & e^{i(\beta-\alpha)} 1_\kR \\ e^{i(\alpha-\beta)} 1_\kR & 0 \end{bmatrix}.$$
$$m_0^* m_0 = \begin{bmatrix} 0 & e^{-i(\alpha-\beta)} 1_\kR \\ e^{-i(\beta-\alpha)} 1_\kR & 0 \end{bmatrix} \begin{bmatrix} 0 & e^{i(\beta-\alpha)} 1_\kR \\ e^{i(\alpha-\beta)} 1_\kR & 0 \end{bmatrix} = \begin{bmatrix} 1_\kR & 0 \\ 0 & 1_\kR \end{bmatrix} = 1_{\mn{2} \otimes \kR}.$$
Since $m_0^* m_0$ is the identity, its operator norm is exactly $\Vert{}m_0\Vert{} = 1$, so $m_0 \in \ball{\kR_{\sfH_{\alpha}}}$. Now, we get
$$\Vert{}\eta(m_0) - \eta(E_{\kR_{\sfH_\beta}}(m_0))\Vert{}_\tau = \vert{}\sin(\beta-\alpha)\vert{} (1) = \vert{}\sin(\beta-\alpha)\vert{}.$$
Therefore, 
$$\sup_{m\in \ball{\kR_{\sfH_{\alpha}}}} \Vert{}\eta(m) - \eta(E_{\kR_{\sfH_{\beta}}}(m))\Vert{}_\tau = \vert{}\sin(\beta-\alpha)\vert{}.$$
\end{proof}

\begin{lem}\label{lem:other-sided-spin}
    Let $\alpha,\beta\in[0,\pi)$ be distinct and $\Theta_{\alpha, \beta}\in \operatorname{Out}(\kR)$ as given by Theorem \ref{thm:bakshi-guin}. Let $\tau$ denote the unique normalized trace on $\mathbb{M}_{2}(\kR)$. Then
    $$\sup_{n\in \ball{\kR_{\sfH_{\beta}}}}\|\eta(n)- \eta(E_{\kR_{\sfH_{\alpha}}}(n))\|_{\tau}= \vert{}\sin(\beta-\alpha)\vert{}.$$
\end{lem}
\begin{proof}
     Let $n \in \ball{\kR_{\sfH_{\beta}}}$. By definition, $n$ has the form
     $$n = \begin{bmatrix} x & y \\ \Theta_{\alpha, \beta} (y) & \Theta_{\alpha, \beta} (x) \end{bmatrix} \quad \text{for } x, y \in \kR.$$
     
     Using the formula in Lemma \ref{lem:condit-exp} for the conditional expectation onto $\kR_{\sfH_{\alpha}}$, we substitute the entries of $n$: $n_{11} = x$, $n_{12} = y$, $n_{21} = \Theta_{\alpha, \beta} (y)$, $n_{22} = \Theta_{\alpha, \beta} (x)$. The diagonal entries of $E_{\kR_{\sfH_{\alpha}}}(n)$ are $\frac{1}{2}(x + \Theta_{\alpha, \beta} (\Theta_{\alpha, \beta} (x))) = x$. For the $(1,2)$ entry, we first calculate the coefficient $b$ before the phase shift
     $$b = \frac{1}{2}\left( e^{i(\alpha-\beta)}y + e^{i(\beta-\alpha)}\Theta_{\alpha, \beta} (\Theta_{\alpha, \beta} (y)) \right) = y \left( \frac{e^{i(\alpha-\beta)} + e^{-i(\alpha-\beta)}}{2} \right) = y \cos(\beta-\alpha).$$
     Applying the structural phase shift for $\kR_{\sfH_{\alpha}}$, the projection becomes
     $$E_{\kR_{\sfH_{\alpha}}}(n) = \begin{bmatrix} x & y \cos(\beta-\alpha) e^{i(\beta-\alpha)} \\ \Theta_{\alpha, \beta} (y) \cos(\beta-\alpha) e^{i(\alpha-\beta)} & \Theta_{\alpha, \beta} (x) \end{bmatrix}.$$
     
      Subtracting the projection $E_{\kR_{\sfH_\alpha}}(n)$ from $n$, we get 
     $$w = \begin{bmatrix} 0 & y - y \cos(\beta-\alpha) e^{i(\beta-\alpha)} \\ \Theta_{\alpha, \beta} (y) - \Theta_{\alpha, \beta} (y) \cos(\beta-\alpha) e^{i(\alpha-\beta)} & 0 \end{bmatrix}.$$
     We can simplify the $(1,2)$ entry by expressing $\cos(\beta-\alpha)$ in terms of complex exponentials or using Euler's formula directly. Letting $\gamma = \beta - \alpha$, we have
     $$1 - \cos(\gamma)e^{i\gamma} = 1 - \cos(\gamma)(\cos(\gamma) + i\sin(\gamma)) = 1 - \cos^2(\gamma) - i\sin(\gamma)\cos(\gamma)$$$$= \sin^2(\gamma) - i\sin(\gamma)\cos(\gamma) = -i\sin(\gamma)(\cos(\gamma) + i\sin(\gamma)) = -i\sin(\gamma)e^{i\gamma}.$$
     A similar simplification for the $(2,1)$ entry yields $i\sin(\gamma)e^{-i\gamma}$. Thus,:
     $$w = \sin(\beta-\alpha) \begin{bmatrix} 0 & -i y e^{i(\beta-\alpha)} \\ i \Theta_{\alpha, \beta} (y) e^{i(\alpha-\beta)} & 0 \end{bmatrix}.$$
     
     Now we get,
     $$w^* w = \sin^2(\beta-\alpha) \begin{bmatrix} \Theta_{\alpha, \beta} (y^*) \Theta_{\alpha, \beta} (y) & 0 \\ 0 & y^* y \end{bmatrix} = \sin^2(\beta-\alpha) \begin{bmatrix} \Theta_{\alpha, \beta} (y^* y) & 0 \\ 0 & y^* y \end{bmatrix}.$$
     Applying the trace on $\mn{2} \otimes \kR$
     $$\tau(w^* w) = \frac{1}{2} \left[ \tau_\kR(\sin^2(\beta-\alpha) \Theta_{\alpha, \beta} (y^* y)) + \tau_\kR(\sin^2(\beta-\alpha) y^* y) \right].$$
     Since $\Theta_{\alpha, \beta} $ is trace-preserving, this gives the $\tau$-norm distance for an arbitrary $n$
     $$\Vert{}\eta(n) - \eta(E_{\kR_u}(n))\Vert{}_\tau = \vert{}\sin(\beta-\alpha)\vert{} \Vert{}y\Vert{}_{\tau_\kR}.$$
     
      As  $n \in \ball{\kR_{\sfH_{\beta}}}$, the operator norm bound $\Vert{}n\Vert{} \le 1$ implies that the $(2,2)$ entry of $n^*n$ satisfies
     $$y^*y + \Theta_{\alpha, \beta} (x^*x) \le 1_\kR.$$
     Taking the trace $\tau_\kR$ yields $\Vert{}y\Vert{}_{\tau_\kR}^2 + \Vert{}x\Vert{}_{\tau_\kR}^2 \le 1$, meaning $\Vert{}y\Vert{}_{\tau_\kR} \le 1$. Therefore, the distance is bounded above by $\vert{}\sin(\beta-\alpha)\vert{}$. To confirm the supremum is attained, we can choose the specific element $n_0 \in \kR_{\sfH_{\beta}}$ by setting $x = 0$ and $y = 1_\kR$ 
     $$n_0 = \begin{bmatrix} 0 & 1_\kR \\ 1_\kR & 0 \end{bmatrix}.$$
    Since, $n_0^* n_0 = 1_{\mn{2} \otimes \kR}$, we have,  $\Vert{}n_0\Vert{} = 1$ and it lies in the $\ball{\kR_{\sfH_{\beta}}}$. Thus, we get 
    $$\Vert{}\eta(n_0) - \eta(E_{\kR_{\sfH_{\alpha}}}(n_0))\Vert{}_\tau = \vert{}\sin(\beta-\alpha)\vert{} \Vert{}1_\kR\Vert{}_{\tau_\kR} = \vert{}\sin(\beta-\alpha)\vert{}.$$
    Therefore, the supremum is exact
    $$\sup_{n\in \ball{\kR_{\sfH_{\beta}}}} \Vert{}\eta(n) - \eta(E_{\kR_{\sfH_{\alpha}}}(n))\Vert{}_\tau = \vert{}\sin(\beta-\alpha)\vert{}.$$
\end{proof}

\begin{theorem}\label{thm:mt-distance-2by2}
 For two distinct $\alpha,\beta\in[0,\pi)$, we have
  \[
  \dmt(\kR_{\sfH_{\alpha}}, \kR_{\sfH_{\beta}})= \vert{}\sin(\alpha-\beta)\vert{}.
\]
\end{theorem}
\begin{proof}
 The proof follows from Lemma \ref{lem:one-sided-spin}, Lemma \ref{lem:other-sided-spin} and Theorem \ref{thm:mashood_taylor}.
\end{proof}

\begin{cor}\label{cor:dmt-is-continuous}
    For every $r \in [0,1]$, there exists $2\times2$ Hadamard unitaries $u$ and $v$ such that
    \[
    \dmt(\kR_u,\kR_v)=r.
    \]
\end{cor}

\begin{cor}\label{cor:dmt-is-same-alpha-2by2}
   For two distinct $\alpha,\beta\in[0,\pi)$, we have
\[
\dmt(\kR_{\sfH_{\alpha}}, \kR_{\sfH_{\beta}})
=
\sqrt{1-\operatorname{cos}\left(\alpha_{\kR_{\sfH_{\alpha}}\cap \kR_{\sfH_{\beta}}}^{\kR}(\kR_{\sfH_{\alpha}},\kR_{\sfH_{\beta}})\right)}.
\]
\end{cor}
\begin{proof}
    The proof follows from the Theorem \ref{thm:mt-distance-2by2} and \cite[Theorem 5.10]{BG2025-1}.
\end{proof}

\begin{cor}\label{cor:angle-not-quantized}
    For every $\theta  \in [0,\frac{\pi}{2}]$, there exist $\alpha, \beta \in [0, \pi)$ such that
    \[
    \alpha^\kR_{\kR_{\sfH_\alpha} \cap \kR_{\sfH_\beta}}(\kR_{\sfH_\alpha},\kR_{\sfH_\beta})=\theta .
    \]
\end{cor}

\begin{cor}\label{cor:comm-sq-d-1-2times2}
   Let $\alpha,\beta\in[0,\pi)$ be distinct. Then the following quadruple
    \[
\begin{array}{ccc}
\kR_{\sfH_\alpha}  & \subset & \kR \\
\cup && \cup \\
\kR_{\sfH_\alpha} \cap \kR_{\sfH_\beta} & \subset & \kR_{\sfH_\beta}
\end{array}
\]
is a commuting square if and only if $\dmt (\kR_{\sfH_\alpha}, \kR_{\sfH_\beta}) = 1.$
\end{cor}
\begin{proof}
    This follows immediately from Corollary \ref{cor:dmt-is-same-alpha-2by2} and \cite[Proposition 2.4]{bakshi-etal-2019}.
\end{proof}

\begin{cor}\label{cor:comm-sq-dkk-1-2times2}
   Let $\alpha,\beta\in[0,\pi)$ be distinct. If the following quadruple is a commuting square
    \[
\begin{array}{ccc}
\kR_{\sfH_\alpha}  & \subset & \kR \\
\cup && \cup \\
\kR_{\sfH_\alpha} \cap \kR_{\sfH_\beta} & \subset & \kR_{\sfH_\beta}
\end{array}
\]
then $\dkk (\kR_{\sfH_\alpha}, \kR_{\sfH_\beta}) = 1.$
\end{cor}

\subsection{The Space of $2\times 2$ Spin Model Subfactors}

\begin{lem}\label{lem:circle}
    Let $X = [0, \pi)$ and define $d: X \times X \to \mathbb{R}$ by $d(\alpha, \beta) = \vert{}\sin(\alpha - \beta)\vert{}$. Then $(X, d)$ is a compact metric space that is isometric to a circle of radius $1/2$ equipped with the Euclidean chordal metric.
\end{lem}
\begin{proof}
    Consider the mapping $f: X \to \mathbb{C}$ defined by $f(\theta) = \frac{1}{2}e^{2i\theta}$. The map $f$ is a bijection from $X$ onto the circle of radius $1/2$ centered at the origin, denoted $C_{1/2}$.
    
    Then any two points $\alpha, \beta \in X$:
    $$
    \vert{}f(\alpha) - f(\beta)\vert{} = \frac{1}{2} \vert{}e^{2i\alpha} - e^{2i\beta}\vert{}
    =\frac{1}{2} \left\vert{} e^{i(\alpha+\beta)} (e^{i(\alpha-\beta)} - e^{-i(\alpha-\beta)}) \right\vert{}
    = \vert{}\sin(\alpha - \beta)\vert{} = d(\alpha, \beta)
    $$
    Therefore $f$ is an isometry between $(X, d)$ and $(C_{1/2}, \vert{}\cdot\vert{})$.
\end{proof}

For each $n\in\mathbb{N}$, let
\[
\mathrm{Sp}_n
:=
\left\{
\kR_{\mathsf{H}}:
\mathsf{H}\in\hn{n}
\right\}
\]
denote the set of all pairwise distinct $n\times n$ spin model subfactors of the hyperfinite $\II_1$-factor $\kR$ induced by Hadamard unitaries in $\hn{n}$.

In the case $n=2$, we have
\[
\mathrm{Sp}_2
=
\{\kR_{\sfH_\alpha}:\alpha\in[0,\pi)\}.
\]
Thus, $\mathrm{Sp}_2$ is the collection of all pairwise distinct $2\times2$ spin model subfactors of $\kR$. We equip $\mathrm{Sp}_2$ with the metric $\dmt$. It then follows immediately from Lemma~\ref{lem:circle} that:

\begin{theorem}\label{thm:sp-2-is-circle}
    The space $\mathrm{Sp}_2$ of all $2\times2$ spin model subfactors, equipped with the metric $\dmt$, is homeomorphic to the circle group $\mathbb{T}$:
\[
(\mathrm{Sp}_2,\dmt)\cong \mathbb{T}.
\]
In particular, $(\mathrm{Sp}_2,\dmt)$ is a compact and path connected metric space.
\end{theorem}

The preceding theorem also shows that $\mathrm{Sp}_2$ naturally admits a group structure. 
Indeed, the group structure of the circle group $\mathbb{T}$ can be transported to 
$\mathrm{Sp}_2$ through the homeomorphism in Theorem~\ref{thm:sp-2-is-circle}. 
For $\alpha,\beta\in[0,\pi)$, define
\[
\alpha\boxplus\beta := (\alpha+\beta)\bmod \pi.
\]
We then define a binary operation $\odot_{\mathrm{sp}}$ on $\mathrm{Sp}_2$ by
\[
\kR_{\sfH_\alpha}\odot_{\mathrm{sp}}\kR_{\sfH_\beta}
:=\kR_{\sfH_{\alpha\boxplus\beta}}.
\]
With respect to $\odot_{\mathrm{sp}}$, $\mathrm{Sp}_2$ becomes a compact, connected
abelian topological group, isomorphic to the circle group $\mathbb{T}$.

\subsection{A Group Structure on $\hn{2}/\bg$}

For each $n\in\mathbb{N}$, define the quotient space
\[
\hn{n}/\bg
:=
\{[\mathsf{H}]_{\bg}:\mathsf{H}\in\hn{n}\},
\]
where $\bg$ denotes the equivalence relation introduced in Definition \ref{def:bakshi-guin-eq-relation}.
In the case $n=2$, it was shown above that
\[
\hn{2}/\bg
=
\{[\mathsf{H}_\alpha]_{\bg}:\alpha\in[0,\pi)\}.
\]
Thus, every $\bg$-equivalence class of $2\times2$ Hadamard unitaries is represented by a unique $\mathsf{H}_\alpha$ with $\alpha\in[0,\pi)$.

Recall that, for $\alpha,\beta\in[0,\pi)$, we define
\[
\alpha\boxplus\beta
:=(\alpha+\beta)\bmod\pi.
\]
This operation allows us to introduce a natural product on $\hn{2}/\bg$. Namely, define
\[
[\mathsf{H}_\alpha]_{\bg}
\odot_{\mathrm{sp}}
[\mathsf{H}_\beta]_{\bg}
:=
[\mathsf{H}_{\alpha\boxplus\beta}]_{\bg},
\qquad
\alpha,\beta\in[0,\pi).
\]
Since each $\bg$-equivalence class has a unique representative of the form $\mathsf{H}_\alpha$ with $\alpha\in[0,\pi)$, this operation is well-defined.

The group structure of $\mathbb{R}/\pi\mathbb{Z}$ now transfers to $\hn{2}/\bg$ through the parametrization
\[
[\alpha]\longmapsto[\mathsf{H}_\alpha]_{\bg}.
\]
In particular, we obtain the following.

\begin{theorem}\label{thm:H2-bg-is-circle}
The quotient space $\hn{2}/\bg$, equipped with the product $\odot_{\mathrm{sp}}$, is an abelian group. Moreover,
\[
\bigl(\hn{2}/\bg,\odot_{\mathrm{sp}}\bigr)
\cong \mathbb{T}.
\]
In particular, $\hn{2}/\bg$ is a compact, connected abelian topological group.
\end{theorem}

\begin{remark}
The operation $\boxplus$ also induces a group structure on the chosen family of representatives
\[
\mathcal{H}_2^{\mathrm{sp}}
:=
\{\mathsf{H}_\alpha:\alpha\in[0,\pi)\}.
\]
More precisely, defining
\[
\mathsf{H}_\alpha\odot_{\mathrm{sp}}\mathsf{H}_\beta
:=
\mathsf{H}_{\alpha\boxplus\beta},
\qquad
\alpha,\beta\in[0,\pi),
\]
endows $\mathcal{H}_2^{\mathrm{sp}}$ with an abelian group structure. The map
\[
\mathsf{H}_\alpha
\longmapsto
e^{2i\alpha}
\]
is then a group isomorphism from $(\mathcal{H}_2^{\mathrm{sp}},\odot_{\mathrm{sp}})$ onto the circle group $\mathbb{T}$. Thus, the same circle-group structure appears naturally at three levels: on the parameter space $[0,\pi)$ equipped with $\boxplus$, on the quotient $\hn{2}/\bg$, and on the corresponding family of representatives of $2\times2$ Hadamard unitaries.
\end{remark}

\subsection{A Family of Regular Spin Model Subfactors}\label{subsec:regularity}

A subalgebra $\sN \subseteq \sM$ of a given von Neumann algebra is said to be
\emph{unitary regular} if
\[
W^*\bigl(\{u\in\sU(\sM):u\sN u^*=\sN\}\bigr)=\sM.
\]
We first show that every Hadamard unitary $u$ satisfying $u\hada\fourier{n}$ gives rise to a unitary regular subfactor $\kR_u\subseteq\kR$.

\begin{prop}\label{prop:uni-reg-spin}
    Let $u$ be an $n\times n$ Hadamard unitary in $\mn{n}$ such that
    $u\hada\fourier{n}$. Then $\kR_u\subseteq\kR$ is unitary regular.
\end{prop}

\begin{proof}
Since $u\hada\fourier{n}$, it follows from \cite[Theorem 5.2]{BGG-2025}
that $\kR_u$ and $\kR_{\fourier{n}}$ are isomorphic. On the other hand,
it is readily verified that
\[
\kR_{\fourier{n}}\rtimes\mathbb{Z}_n\cong\kR.
\]
Consequently,
\[
(\kR_u\subseteq\kR)
\cong
(\kR_{\fourier{n}}\subseteq\kR)
\cong
(\kR_{\fourier{n}}\subseteq
\kR_{\fourier{n}}\rtimes\mathbb{Z}_n),
\]
and the latter inclusion is unitary regular. Hence, $\kR_u\subseteq\kR$
is unitary regular.
\end{proof}

\begin{lem}\label{lem:uncountable-unitaries}
    For any integer $n \ge 2$, there exists an uncountable family of Hadamard unitaries $\{ u_\lambda \}_{\lambda \in I}$ in $\mn{n}$ such that:
    \begin{itemize}
        \item[(i)] $u_\lambda \hada \fourier{n}$ for all $\lambda \in I$.
        \item[(ii)] $u_\lambda \not\bg \fourier{n}$ for all $\lambda \in I$.
        \item[(iii)] $u_\lambda \not\bg u_\mu$ for all $\lambda \neq \mu$ in $I$.
    \end{itemize}
\end{lem}

\begin{proof}

    Let $\Omega_n = \{ 1, \omega, \omega^2, \dots, \omega^{n-1} \}$ denote the set of $n$-th roots of unity.
    
    Define the index set $I$ to be the open arc on the unit circle given by:
    $$I := \{ e^{i\theta} \mid 0 < \theta < 2\pi/n \}$$
    Note that $I$ contains no $n$-th roots of unity.

    For each $\lambda \in I$, define the diagonal unitary matrix $D_\lambda = \operatorname{diag}(1, \lambda, 1, \dots, 1)$ and construct the Hadamard unitary:
    $$u_\lambda = D_\lambda \fourier{n}.$$
    Since $\fourier{n}$ is a Hadamard unitary and $u_\lambda$ is obtained from $\fourier{n}$ by left-multiplying by a diagonal unitary matrix, $u_\lambda$ is also a Hadamard unitary. By definition of Hadamard equivalence, this immediately establishes that $u_\lambda \hada \fourier{n}$ for all $\lambda \in I$, proving property $(i)$.

    To prove property $(ii)$, assume for the sake of contradiction that $u_\lambda \bg \fourier{n}$ for some $\lambda \in I$. By the definition of Bakshi--Guin equivalence, there exist an $n \times n$ permutation matrix $P$ and a diagonal unitary matrix $D = \operatorname{diag}(d_0, d_1, \dots, d_{n-1})$ such that
    $$u_\lambda = \fourier{n} P D.$$

    Substitute $u_\lambda = D_\lambda \fourier{n}$. The $(0,0)$-entry of $D_\lambda$ is $1$, so the $0$-th row of $u_\lambda$ is identical to the $0$-th row of $\fourier{n}$, which is $\frac{1}{\sqrt{n}}(1, 1, \dots, 1)$. On the right side, multiplying $\fourier{n}$ by the permutation matrix $P$ leaves the uniform $0$-th row unchanged. Multiplying by $D$ scales the $k$-th entry of the $0$-th row by $d_k$. Equating the $0$-th rows of both sides gives
    $$\frac{1}{\sqrt{n}}(1, 1, \dots, 1) = \frac{1}{\sqrt{n}}(d_0, d_1, \dots, d_{n-1}).$$
    This implies $d_k = 1$ for all $k$, so $D$ is the identity matrix $\mathbb{I}_n$. The equation reduces to $u_\lambda = \fourier{n} P$.

    Now compare the $1$-st row of both sides. The $1$-st row of $u_\lambda = D_\lambda \fourier{n}$ is multiplied by $\lambda$, so its entries form the set $\{ \frac{\lambda}{\sqrt{n}}\omega^k \mid 0 \le k \le n-1 \}$. The $1$-st row of $\fourier{n} P$ is a permutation of the $1$-st row of $\fourier{n}$, so its entries form the set $\{ \frac{1}{\sqrt{n}}\omega^k \mid 0 \le k \le n-1 \}$. Since $u_\lambda = \fourier{n} P$, these two sets of entries must be equal. However, this requires $\lambda = \omega^k$ for some integer $k$, which implies $\lambda \in \Omega_n$. This contradicts our construction that $I \cap \Omega_n = \emptyset$. Thus, $u_\lambda \not\bg \fourier{n}$ for all $\lambda \in I$, proving property $(ii)$.

    To prove property $(iii)$, suppose that $u_\lambda \bg u_\mu$ for some $\lambda, \mu \in I$. By definition, there exist a permutation matrix $P$ and a diagonal unitary matrix $D = \operatorname{diag}(d_0, d_1, \dots, d_{n-1})$ such that:
    $$u_\lambda = u_\mu P D$$

    Substitute the definitions of $u_\lambda$ and $u_\mu$:

    $$D_\lambda \fourier{n} = D_\mu \fourier{n} P D$$

    Consider the $0$-th row of both sides. Since the $(0,0)$-entries of both $D_\lambda$ and $D_\mu$ are $1$, the $0$-th rows of $D_\lambda \fourier{n}$ and $D_\mu \fourier{n}$ are both precisely $\frac{1}{\sqrt{n}}(1, 1, \dots, 1)$. Right-multiplying $D_\mu \fourier{n}$ by $P$ leaves this uniform row unchanged, and right-multiplying by $D$ scales the $k$-th entry by $d_k$. Equating the $0$-th rows of both sides yields:

    $$\frac{1}{\sqrt{n}}(1, 1, \dots, 1) = \frac{1}{\sqrt{n}}(d_0, d_1, \dots, d_{n-1})$$

    This forces $d_k = 1$ for all $k$, so $D = \mathbb{I}_n$. The equivalence simplifies to
    $$u_\lambda = u_\mu P.$$

    Now, consider the $1$-st row of this simplified equation. The $1$-st row of $u_\lambda$ consists of the entries $\frac{\lambda}{\sqrt{n}}\omega^k$ for $0 \le k \le n-1$. The $1$-st row of $u_\mu P$ is a permutation of the $1$-st row of $u_\mu$, which consists of the entries $\frac{\mu}{\sqrt{n}}\omega^k$ for $0 \le k \le n-1$. Because $u_\lambda = u_\mu P$, the sets of entries in their $1$-st rows must be identical

    $$\left\{ \frac{\lambda}{\sqrt{n}} \omega^k \;\middle\vert{}\; 0 \le k \le n-1 \right\} = \left\{ \frac{\mu}{\sqrt{n}} \omega^k \;\middle\vert{}\; 0 \le k \le n-1 \right\}.$$

    This set equality requires that the specific element $\frac{\lambda}{\sqrt{n}} \cdot 1$ from the left set must equal some element in the right set. Thus, there exists an integer $k$ ($0 \le k \le n-1$) such that

    $$\lambda = \mu \omega^k \implies \lambda \mu^{-1} = \omega^k.$$

    Since $\lambda, \mu \in I$, we can write $\lambda = e^{i\theta_1}$ and $\mu = e^{i\theta_2}$ for some angles $\theta_1, \theta_2 \in (0, 2\pi/n)$. The ratio $\lambda \mu^{-1} = e^{i(\theta_1 - \theta_2)}$ has an argument $\theta_1 - \theta_2$ that strictly satisfies:

    $$-\frac{2\pi}{n} < \theta_1 - \theta_2 < \frac{2\pi}{n}.$$

    The only $n$-th root of unity $\omega^k = e^{2\pi i k / n}$ whose argument falls strictly within the open interval $(-2\pi/n, 2\pi/n)$ is $\omega^0 = 1$.

    Therefore, we must have $\lambda \mu^{-1} = 1$, which implies $\lambda = \mu$. By contraposition, if $\lambda \neq \mu$ in $I$, then $u_\lambda \not\bg u_\mu$. This establishes property $(iii)$ and completes the proof.
    
\end{proof}

For von Neumann subalgebras of a given von Neumann algebra, there is a more general notion of \emph{regularity} (see, e.g., \cite{Renault2008CartanSI,Kumjian_86,bakshi-ghosh-2026}). It is straightforward to verify that every unitarily regular subalgebra is regular; however, the converse does not hold in general (see \cite{BGK2026} for further details). Recently, Bakshi and Ghosh proved that, for a subfactor inclusion $\sN \subseteq \sM$ in which $\sN$ and $\sM$ are of the same type, these two notions of regularity are equivalent; that is, every regular subfactor is unitarily regular \cite{bakshi-ghosh-2026}. Thus, for a subfactor inclusion $\sN \subseteq \sM$ of the same type, we shall simply use the term \emph{regular}, without ambiguity.

With this convention in mind, we have the following result.
\begin{prop}\label{prop:uncountable-reg-spin-models}
    For each $n \in \N$, the hyperfinite $\II_1$-factor $\kR$ contains an uncountable family of pairwise distinct regular $n \times n$ spin model subfactors.
\end{prop}
\begin{proof}
Let $\mathcal{U}$ be an uncountable family of Hadamard unitaries in $\mn{n}$ as provided by Lemma \ref{lem:uncountable-unitaries}. By condition~(iii) of the lemma and \cite[Theorem 4.7]{BG2025-1}, the spin model subfactors associated with the unitaries in $\mathcal{U}$ are pairwise distinct. On the other hand, condition~(i) ensures that every $u \in \mathcal{U}$ is Hadamard equivalent to $\fourier{n}$. Therefore, Proposition \ref{prop:uni-reg-spin} implies that each of the corresponding spin model subfactors is regular. Hence, $\mathcal{U}$ gives rise to an uncountable family of pairwise distinct regular $n \times n$ Spin model subfactors of $\kR$.
\end{proof}

    \section{Distances and Commuting Square}
    \label{sec:dist-comm-sq}

Motivated by Corollaries \ref{cor:comm-sq-d-1-2times2} and \ref{cor:comm-sq-dkk-1-2times2}, in this section, we study to what extent the maximal interior angle between two intermediate $\II_1$-subfactors (that is, a commuting square) can capture the maximal distance between them.

We begin with the Remark \ref{rem:comm-sq-vs-max-d}, which establishes that, in general, the notion of a commuting square for a quadruple is not related to the maximal values of the distance. Then, after proving an important observation on the inclusions of finite-index $\II_1$-factors (see Proposition \ref{prop:exist-exp-zero-uni}), we are able to prove that for the inclusions of finite-index subfactors when exactly a commuting square (maximal interior angle) maximizes the two distances (see Theorem \ref{thm:comm-sq-imply-d-1-ii-1} and Theorem \ref{thm:dkk-is-1-for-comm-sq-spin}). We also prove that the intermediate spin model subfactors which form a commuting square in the hyperfinite $\II_1$-factor $\kR$ have maximal distances (see Theorem \ref{thm:spin-model-n-times-n}).

In Theorem \ref{thm:diffuse-comm-sq}, we prove that two diffuse von Neumann subalgebras that are orthogonal in the sense of Popa are necessarily at maximal distance.

\begin{remark}\label{rem:comm-sq-vs-max-d}
The phenomenon observed in Corollary \ref{cor:comm-sq-d-1-2times2} and \ref{cor:comm-sq-dkk-1-2times2} does not hold for a general inclusion of $\II_1$ factors $\sN \subseteq \sM$. In particular, two intermediate subfactors may attain the maximal Kadison--Kastler distance without forming a commuting square. We illustrate this with the following example.

Let $G$ be a finite group acting outerly on the hyperfinite $\II_1$-factor $\kR$, and let $H,K$ be two non-trivial distinct subgroups of $G$ such that $H \cap K \ne \{e \}$. Then it is well known that the following quadruple is not a commuting square:
\[ \begin{array}{ccc} \kR \rtimes H & \subset & \kR \rtimes G \\ \cup && \cup \\ \kR & \subset & \kR \rtimes K \end{array} \]
Hence,
\[
\alpha(\kR\rtimes H,\kR\rtimes K)\neq \frac{\pi}{2},
\]
and consequently, by \cite[Proposition 2.4]{bakshi-etal-2019}, the above square is not a commuting square. On the other hand, \cite[Theorem 5.11]{kumar-2026} implies that
\[
\dmt(\kR\rtimes H,\kR\rtimes K)
=
1
=
\dkk(\kR\rtimes H,\kR\rtimes K).
\]
Thus, the two intermediate subfactors are at maximal Kadison--Kastler distance even though they do not form a commuting square.

For the converse direction, namely, whether the formation of a commuting square forces the Kadison--Kastler distance to be maximal, we provide a sufficient condition in Theorem \ref{thm:comm-sq-imply-d-1-ii-1}.
\end{remark}

\begin{prop}\label{prop:exist-exp-zero-uni}
    Let $\mathcal{N} \subseteq \mathcal{M}$ be an inclusion of $\mathrm{II}_1$-subfactors with $[\mathcal{M}:\mathcal{N}] \in \{2,3\} \cup [4,\infty]$. Then there exists a unitary $u \in \mathcal{M}$ such that $E_{\mathcal{N}}(u) = 0$ where $E_\sN$ be the canonical trace preserving conditional expectation onto $\sN$.
\end{prop}

\begin{proof}
    Following Popa \cite{popa-1989-rel-dim}, we define the set of relative dimensions $\Lambda(\mathcal{M}, \mathcal{N})$ as the set of all scalars $\alpha \in [0, 1]$ for which there exists a projection $p \in \mathcal{M}$ such that $E_{\mathcal{N}}(p) = \alpha 1_{\mathcal{N}}$. 
    
    If we can show that $1/2 \in \Lambda(\mathcal{M}, \mathcal{N})$, then there exists a projection $p \in \mathcal{M}$ satisfying $E_{\mathcal{N}}(p) = \frac{1}{2} 1_{\mathcal{N}}$. We can then define the element $u = 2p - 1$. Because $p$ is a projection, $u$ is a self-adjoint unitary in $\mathcal{M}$. Applying the conditional expectation yields
    $$E_{\mathcal{N}}(u) = 2 E_{\mathcal{N}}(p) - 1_{\mathcal{N}} = 2\left(\frac{1}{2} 1_{\mathcal{N}}\right) - 1_{\mathcal{N}} = 0,$$
    which would complete the proof. Thus, it suffices to prove that $1/2 \in \Lambda(\mathcal{M}, \mathcal{N})$ for all index values in the given range. Let $\lambda = [\mathcal{M}:\mathcal{N}]^{-1}$.

    \textbf{Case 1:} Assume $[\mathcal{M}:\mathcal{N}] \in [4, \infty]$. 
    By Theorem 5.2 of \cite{popa-1989-rel-dim}, setting the parameter $t = 1/2$ directly implies that $1/2 \in \Lambda(\mathcal{M}, \mathcal{N})$ for any subfactor with index greater than or equal to $4$.

    \textbf{Case 2:} Assume $[\mathcal{M}:\mathcal{N}] \in \{2, 3\}$.
    By the index rigidity theorem, these values correspond to the discrete part of the Jones spectrum, where the index takes the form $4\cos^2(\pi/n)$ for $n=4$ and $n=6$, respectively. By Theorem 5.1 of \cite{popa-1989-rel-dim}, the set of relative dimensions for this discrete spectrum is explicitly given by:
    $$\Lambda(\mathcal{M}, \mathcal{N}) = \{0\} \cup \left\{ \lambda \frac{P_{k-1}(\lambda)}{P_k(\lambda)} \;\middle|\; 1 \le k \le n-2 \right\},$$
    where $P_k(x)$ are polynomials defined recursively by $P_0(x) = 1$, $P_1(x) = 1$, and $P_{k+1}(x) = P_k(x) - x P_{k-1}(x)$.
    \begin{itemize}
        \item For $[\mathcal{M}:\mathcal{N}] = 2$ (which means $n=4$ and $\lambda = 1/2$), evaluating the polynomials for $k=1, 2$ explicitly yields $\Lambda(\mathcal{M}, \mathcal{N}) = \{0, 1/2, 1\}$. Clearly, $1/2 \in \Lambda(\mathcal{M}, \mathcal{N})$.
        \item For $[\mathcal{M}:\mathcal{N}] = 3$ (which means $n=6$ and $\lambda = 1/3$), evaluating the polynomials up to $k=4$ yields the set of relative dimensions $\{0, 1/3, 1/2, 2/3, 1\}$. Again, $1/2 \in \Lambda(\mathcal{M}, \mathcal{N})$.
    \end{itemize}

    In all cases, we have established that $1/2 \in \Lambda(\mathcal{M}, \mathcal{N})$. The existence of the trace-zero unitary $u = 2p - 1$ immediately follows.
\end{proof}

\begin{theorem}\label{thm:dkk-is-1-for-comm-sq-spin}
Let $\sM$ be a finite von Neumann algebra, and let $\sN,\sP,\sQ$ be $\II_1$-subfactors of $\sM$ such that
\[
\begin{array}{ccc}
    \sP & \subset & \sM \\
    \cup && \cup \\
    \sN & \subset & \sQ
\end{array}
\]
is a commuting square. Suppose that the index of at least one of the inclusions $\sN\subseteq\sP$ or $\sN\subseteq\sQ$ belongs to $\{2,3\}\cup[4,\infty]$. Then $\dkk(\sP,\sQ)=1=\dmt(\sP,\sQ)$.
\end{theorem}
\begin{proof}
Without loss of generality, assume that $[\sP: \sN]\in \{2,3\}\cup [4, \infty]$. 
Let $E^{\sM}_{\sN}: \sM \to \sN$, $E^{\sM}_{\sP}: \sM \to \sP$ and $E^{\sP}_{\sN}: \sP \to \sN$ be the unique trace-preserving conditional expectations. Then, for any $x\in \sM$, we have
\[
E^{\sP}_{\sN} \circ E^{\sM}_{\sP}(x) = E^{\sM}_{\sN}(x).
\]
Let $x\in \ball{P}$. Then
\[
\|\eta(x)-\eta(E^{\sM}_{\sQ}(x))\|^2_{\tau}
= \|\eta(x)-\eta(E^{\sM}_{\sQ}(E^{\sM}_{\sP}(x)))\|^2_{\tau}
= \|\eta(x-E^{\sM}_{\sN}(x))\|^2_{\tau}
= \|\eta(x)\|^2_{\tau}-\|\eta(E^{\sM}_{\sN}(x))\|^2_{\tau}
\leq 1.
\]
This implies that
\[
\sup_{x \in \ball{P}}\|\eta(x)-\eta(E^{\sM}_{\sQ}(x))\|_{\tau}\leq 1.
\]

Now, by Proposition \ref{prop:exist-exp-zero-uni}, there exists a unitary $u\in \sP$ such that $E^{\sP}_{\sN}(u)=0$. Hence,
\[
\|\eta(u)-\eta(E^{\sM}_{\sQ}(u))\|^2_{\tau}
= \|\eta(u)\|^2_{\tau}-\|\eta(E^{\sM}_{\sN}(u))\|^2_{\tau}
= \|\eta(u)\|^2_{\tau}-\|\eta(E^{\sP}_{\sN}\circ E^{\sM}_{\sP}(u))\|^2_{\tau}
=1.
\]
Therefore,
\[
\sup_{x \in \ball{P}}\|\eta(x)-\eta(E^{\sM}_{\sQ}(x))\|_{\tau}=1.
\]
The result now follows from Theorem \ref{thm:mashood_taylor} and Proposition \ref{prop:mt_and_kk}.
\end{proof}

\begin{theorem}\label{thm:comm-sq-imply-d-1-ii-1}
    Let $\sN \subseteq \sM$ be $\mathrm{II}_1$ subfactors with $[\mathcal{M}:\mathcal{N}]\ge 4$, and let $\sP,\sQ$ be intermediate $\mathrm{II}_1$ subfactors.
    Suppose that
    \[
        \begin{array}{ccc}
            \sP & \subset & \sM \\
            \cup & & \cup \\
            \sN & \subset & \sQ
        \end{array}
    \]
    is a commuting square. Then
    \[
        \dmt(\sP,\sQ)=1=\dkk(\sP,\sQ).
    \]
\end{theorem}

\begin{proof}
    Let $E^{\sM}_{\sN}: \sM \to \sN$, $E^{\sM}_{\sP}: \sM \to \sP$ and $E^{\sP}_{\sN}: \sP \to \sN$ be the unique trace-preserving conditional expectations. Then, for any $x\in \sM$, we have
\[
E^{\sP}_{\sN} \circ E^{\sM}_{\sP}(x) = E^{\sM}_{\sN}(x).
\]
Let $x\in \ball{P}$. Then
\[
\|\eta(x)-\eta(E^{\sM}_{\sQ}(x))\|^2_{\tau}
= \|\eta(x)-\eta(E^{\sM}_{\sQ}(E^{\sM}_{\sP}(x)))\|^2_{\tau}
= \|\eta(x-E^{\sM}_{\sN}(x))\|^2_{\tau}
= \|\eta(x)\|^2_{\tau}-\|\eta(E^{\sM}_{\sN}(x))\|^2_{\tau}
\leq 1.
\]
This implies that
\[
\sup_{x \in \ball{P}}\|\eta(x)-\eta(E^{\sM}_{\sQ}(x))\|_{\tau}\leq 1.
\]

Now, by Proposition \ref{prop:exist-exp-zero-uni}, there exists a unitary $u\in \sP$ such that $E^{\sP}_{\sN}(u)=0$. Hence,
\[
\|\eta(u)-\eta(E^{\sM}_{\sQ}(u))\|^2_{\tau}
= \|\eta(u)\|^2_{\tau}-\|\eta(E^{\sM}_{\sN}(u))\|^2_{\tau}
= \|\eta(u)\|^2_{\tau}-\|\eta(E^{\sP}_{\sN}\circ E^{\sM}_{\sP}(u))\|^2_{\tau}
=1.
\]
Therefore,
\[
\sup_{x \in \ball{P}}\|\eta(x)-\eta(E^{\sM}_{\sQ}(x))\|_{\tau}=1.
\]
The result now follows from Theorem \ref{thm:mashood_taylor} and Proposition \ref{prop:mt_and_kk}.
\end{proof}

\begin{theorem}\label{thm:spin-model-n-times-n}
    Let $u,v\in\hn{n}$ with $u\not\bg v$ such that $\kR_u\cap\kR_v$ is a finite-index subfactor of $\kR$.
    If
    \[
        \begin{array}{ccc}
            \kR_u & \subset & \kR \\
            \cup && \cup \\
            \kR_u\cap\kR_v & \subset & \kR_v
        \end{array}
    \]
    is a commuting square, then
    \[
        \dkk(\kR_u,\kR_v)=1=\dmt(\kR_u, \kR_v).
    \]
\end{theorem}
\begin{proof}
    We have

    \[
    [\kR:\kR_u\cap\kR_v]=\begin{cases}
        4 &\text{ if } n = 2\\
       9 &\text{ if } n = 3\\
       \ge 4 & \text{ if } n \ge 4.
    \end{cases}
    \]
    The conclusion now follows immediately from Theorem~\ref{thm:comm-sq-imply-d-1-ii-1}.
\end{proof}

\begin{cor}\label{cor:necessity-for-commuting-square}
Let $u,v\in\hn{n}$ with $u\not\bg v$ and $\dkk(\kR_u,\kR_v)\ne 1$. Then either $\kR_u\cap\kR_v$ is not a finite-index subfactor of $\kR$, or the quadruple $(\kR,\kR_u,\kR_v,\kR_u\cap\kR_v)$ does not form a commuting square.
\end{cor}


\begin{definition}[Orthogonality, \cite{popa-1983}]
 Let $\sM$ be a finite von Neumann algebra with a faithful normal tracial state $\tau$, and let $\sB_1,\sB_2$ be von Neumann subalgebras of $\sM$. We say that $\sB_1$ is orthogonal to $\sB_2$ if $\tau(b_1b_2)=0$ whenever $b_1\in \sB_1$ and $b_2\in \sB_2$ satisfy $\tau(b_1)=\tau(b_2)=0$. 
\end{definition}

\begin{remark}\label{rem:popa-ortho-not-same}
    Note that, by \cite[Lemma 2.1]{popa-1983}, if $\sP$ and $\sQ$ are orthogonal in $\sM$ in the sense of Popa, then $\sP \cap \sQ = \C$ and $(\C,\sP,\sQ,\sM)$ forms a commuting square. On the other hand, by \cite{BG2025-1}, the intersection $\kR_{\sfH_{\alpha}} \cap \kR_{\sfH_{\beta}}$ is a $\II_1$ subfactor and, in particular, can never be equal to $\C$. Thus, for any $\alpha,\beta \in [0,\pi)$, the Spin model subfactors $\kR_{\sfH_{\alpha}}$ and $\kR_{\sfH_{\beta}}$ can never be orthogonal inside the hyperfinite $\II_1$-factor in the sense of Popa.

On the other hand, Corollary~\ref{cor:dmt-is-same-alpha-2by2} shows that whenever the distance between $\kR_{\sfH_{\alpha}}$ and $\kR_{\sfH_{\beta}}$ is maximal, their interior angle over their intersection is maximal, namely,
\[
\alpha_{\kR_{\sfH_{\alpha}} \cap \kR_{\sfH_{\beta}}}^\kR
\bigl(\kR_{\sfH_{\alpha}},\kR_{\sfH_{\beta}}\bigr)
= \frac{\pi}{2}.
\]
Thus, orthogonality in terms of the interior angle is distinct from, and does not imply, orthogonality in the sense of Popa.
\end{remark}

\begin{theorem}\label{thm:diffuse-comm-sq}
Let $(\sM, \tau)$ be a finite von Neumann algebra, and let $\sP,\sQ\subseteq \sM$ be diffuse orthogonal subalgebras. Then
\[
\dkk(\sP,\sQ)=1=\dmt(\sP,\sQ).
\]
\end{theorem}
\begin{proof}
    Since $\sP$ and $\sQ$ are orthogonal then by Lemma 2.1-((i) $\implies$ (iv)) of \cite{popa-1983}, we get $E_\sQ E_\sP = \tau(-) 1_\sM$. Therefore, for any $x \in \sP$, applying $E_\sQ$ yields:
    $$E_\sQ(x) = E_\sQ(E_\sP(x)) =  \tau(x)1_\sM$$
    
    Because $\sP$ is a diffuse von Neumann algebra, it contains a Haar unitary $u$. This unitary satisfies $u \in \ball{\sP}$, $\Vert{}u\Vert{}_\tau = 1$, and $\tau(u) = 0$. 
    
    Then, we get 
    $$\Vert{}\eta(u - E_\sQ(u))\Vert{}_\tau = \Vert{}\eta(u - \tau(u)1)\Vert{}_\tau = \Vert{}\eta(u - 0)\Vert{}_\tau = 1.$$
    
    Since the supremum over $\ball{\sP}$ cannot exceed $1$ (as $\Vert{}\eta(x - E_\sQ(x))\Vert{}_\tau \le \Vert{}\eta(x)\Vert{}_\tau \le 1$), we immediately obtain
    $$\sup_{x \in \ball{\sP}} \Vert{}\eta(x - E_\sQ(x))\Vert{}_\tau = 1.$$
    Then by Theorem \ref{thm:mashood_taylor} and Proposition \ref{prop:mt_and_kk}, we get 
    $$\dkk(\sP, \sQ)= \dmt(\sP, \sQ) = 1.$$
\end{proof}

For two diffuse subalgebras $\sP, \sQ$, the claim that $\dmt(\sP, \sQ) = 1$ implies orthogonality is not true. Below is a concrete counterexample.

\begin{example} \label{exam:ortho-converse-not-true}
Let $\sM = L^\infty(\mathbb{T}^3, \mu \times \mu \times \mu)$, where $\mathbb{T} = \R / 2\pi\Z$ and $\mu$ is the normalized Haar measure. We equip $\sM$ with the standard tracial state $\tau$ given by integration over $\mathbb{T}^3$.

Define the following von Neumann subalgebras:
\begin{align*}
\sP &= \{ f \in \sM \mid f(r, s, t) = g(r, s) \text{ for some } g \in L^\infty(\mathbb{T}^2) \} \cong L^\infty(\mathbb{T}^2) \otimes 1,\\
\sQ &= \{ f \in \sM \mid f(r, s, t) = h(s, t) \text{ for some } h \in L^\infty(\mathbb{T}^2) \} \cong 1 \otimes L^\infty(\mathbb{T}^2).
\end{align*}
Both $\sP$ and $\sQ$ are diffuse, unital subalgebras of $\sM$. We will show that $\dmt(\sP, \sQ) = 1$, but $\sP$ and $\sQ$ are not orthogonal.
\end{example}

\begin{proof}
Notice that $\sP \cap \sQ$ contains all functions that depend only on the middle variable $s$. 
Let $a(r,s,t) = \sqrt{2}\cos(s)$. Since $a$ depends only on $s$, we have $a \in \sP$ and $a \in \sQ$.

Then,
\[
\tau(a) = \int_{\mathbb{T}} \sqrt{2}\cos(s) \,d\mu(s) = 0.
\]
Therefore we have
\[
\tau(aa) = \tau(a^2) = \int_{\mathbb{T}} 2\cos^2(s) \,d\mu(s) = 1 \neq 0= \tau(a) \tau(a).
\]

\noindent\textbf{Claim: }$\dmt(\sP, \sQ) = 1$.\\
Let $u(r,s,t) = e^{ir}$. Clearly, $u \in \sP$ and $\|u\|_\infty = 1$, so $u \in \ball{\sP}$. Applying $E_\sQ$, we get:
\[
E_\sQ(u)(s,t) = \int_{\mathbb{T}} e^{ir} \,d\mu(r) = 0.
\]
Consequently, $\|u - E_\sQ(u)\|_2 = \|u\|_2 = 1$.
Since $\|x - E_\sQ(x)\|_2 \leq \|x\|_2 \leq 1$ for all $x \in \ball{\sP}$, we conclude that
\[
\sup_{x\in \ball{\sP}}\|\eta(x)-\eta(E_{\sQ}(x))\|_{\tau} = 1.
\]

By symmetry, choosing $v(r,s,t) = e^{it} \in \ball{\sQ}$ yields $E_\sP(v) = 0$ and $\|v - E_\sP(v)\|_2 = 1$, which gives
\[
\sup_{z\in \ball{\sQ}}\|\eta(z)-\eta(E_{\sP}(z))\|_{\tau} = 1.
\]

By Theorem \ref{thm:mashood_taylor}, we conclude that $\dmt(\sP, \sQ)= 1$.
\end{proof}

    \section{Distances Between Conjugate Masas in Free Group Factors}
    \label{sec:gr-vNa-dist}

In this section, we shift our focus to the rich, non-commutative geometric landscape of group von Neumann algebras, specifically within the free group factor $L(\mathbb{F}_2)$. While subgroups naturally generate von Neumann subalgebras, the full lattice of subalgebras is vastly larger and often heavily decoupled from the underlying group structure. To rigorously quantify the relative positions and structural differences of these algebras, we must first develop the framework of Fourier analysis on discrete groups. 

The primary purpose of this section is to leverage this Fourier-analytic machinery to explicitly compute the Mashood--Taylor distance between the canonical masa $L(\langle a \rangle)$ and its unitary conjugates $u L(\langle a \rangle) u^*$ (see Theorem \ref{thm:mt_is_trace}). By analyzing the Fourier coefficients of these unitaries, we will prove that this geometric distance continuously attains values in the entire interval $[0,1]$ (see Corollary \ref{cor:mt-attains-val}). Furthermore, we culminate the section with a striking structural result: we construct maximally distant conjugate masas (where the distance is exactly $1$) that do not originate from any subgroup of $\mathbb{F}_2$, highlighting the purely analytic flexibility of these inclusions (see Theorem \ref{thm:dkk-1-for-non-tiv-alg}).

\subsection{Fourier Theory on Discrete Group}

Let $G$ be a countable discrete group. We define the von Neumann algebra associated with $G$ as follows:
\[
L(G):= \{\lambda_g: g\in G \}'' \subseteq \sB(\ell^2(G))
\]

For $x\in L(G)$ we define a complex valued function $\widehat{x}$ on $G$ by
\[
\widehat{x}(g):=\tau(\lambda_g^* x)\,\quad (g\in G).
\]
We define $\widehat{L(G)}:= \{ \widehat{x}: x\in L(G) \}$.

\begin{lem}\label{lem:fourier_coeff}
    Let $G$ be a countable discrete group. Let $\Phi: \eta(L(G)) \to \ell^{2}(G)$
    be the linear isometry defined by
    \[
        \Phi(\eta(x)) := x\delta_e, \qquad x\in L(G).
    \]
    Then the following hold:
    \begin{enumerate}
        \item[(i)] $\Phi$ extends to a unitary operator from $L^2(L(G),\tau)$ onto $\ell^2(G)$. Consequently,
        \[
            \{\eta(\lambda_g):g\in G\}
        \]
        is an orthonormal basis for $L^2(L(G),\tau)$.

        \item[(ii)] For every $x\in L(G)$, the $g$-th Fourier coefficient $\langle \eta(x),\eta(\lambda_g)\rangle_{L^2}$ of $\eta(x)$ is equal to $\widehat{x}(g).$ Consequently,
        \begin{align*}
            \eta(x)
            = \sum_{g\in G}\widehat{x}(g)\eta(\lambda_g),\quad \text{ and }\quad
            \eta(x^*)
            = \sum_{g\in G}
            \overline{\widehat{x}(g)}
            \eta(\lambda_{g^{-1}}),
        \end{align*}
        where both series converge in the Hilbert space $L^2(L(G),\tau)$.
    \end{enumerate}
\end{lem}
\begin{proof}
\noindent \textbf{(i).}  Since  the range of $\Phi$ is dense and it is isometry so it extends to an unitary from $L^2(L(G),\tau)$ onto $\ell^2(G)$.\\  
   \noindent \textbf{(ii).} First note that $\tau(\lambda_g^*x)$ is equal to the $g$-th Fourier coefficient of $x\delta_e$ for any $x \in L(G)$. Indeed,
   \[
   \langle x\delta_e, \delta_g \rangle_{\ell^2}= \langle x\delta_e, \lambda_g \delta_e \rangle_{\ell^2}=\langle \lambda_g^* x \delta_e, \delta_e \rangle_{\ell^2}= \tau(\lambda_g^* x).
   \]
   Hence we get
   \begin{align*}
       \widehat{x} = \Phi^{-1} (x\delta_e) = \Phi^{-1} \left( \sum_{g\in G} \tau(\lambda_g^* x) \delta_g   \right) = \sum_{g\in G} \tau(\lambda_g^* x) \Phi^{-1}(\delta_g) = \sum_{g\in G} \tau(\lambda_g^* x) \Phi^{-1} (\lambda_g \delta_e) = \sum_{g\in G} \tau(\lambda_g^* x) \widehat{\lambda_g}
   \end{align*}
\end{proof}

\begin{lem}\label{lem:mult_is_convolution}
Let $G$ be a countable discrete group. Then for all $x \in L(G)$, we have $\widehat{x} \in \ell^2(G)$. Moreover, $\ell^1(G) \subseteq \widehat{L(G)} \subseteq \ell^2(G)$. Consequently, for all $x, y \in L(G)$ and for every $g \in G$, the convolution sum defining $(\widehat{x} \ast \widehat{y})(g)$ converges absolutely, and we have
$$\widehat{x} * \widehat{y} = \widehat{xy}.$$
In particular, $\widehat{x} \ast \widehat{y} \in \ell^2(G)$.
\end{lem}
\begin{proof}
From Lemma \ref{lem:fourier_coeff}-(ii) we get that $\widehat{x}(g) = \langle \eta(x), \eta(\lambda_g)\rangle_{L^2}$. Thus 
    \[
    \sum_g |\widehat{x}(g)|^2 = \|\eta(x)\|_\tau^2 <\infty.
    \]
     The convolution of $\widehat{x}$ and $\widehat{y}$ at $g \in G$ is defined as
     $$(\widehat{x} * \widehat{y})(g) = \sum_{h \in G} \widehat{x}(h) \widehat{y}(h^{-1}g).$$
    By Cauchy-Schwartz inequality, this convolution sum converges absolutely, i.e., $(\widehat{x} * \widehat{y})(g)$ exists for all $g \in G$. 
    
    Now for any $x,y \in L(G)$ and $g\in G$, we have 

    \begin{align*}
        \widehat{xy}(g) = \langle \eta(xy), \eta(\lambda_g)\rangle_{L^2}= \langle xy \delta_e, \delta_g \rangle_{\ell^2}
    &=\langle y \delta_e, x^*\delta_g\rangle_{\ell^2}\\
    &= \sum_{h \in G} \langle y \delta_e, \delta_h \rangle_{\ell^2} \langle  \delta_h, x^*\delta_g \rangle_{\ell^2}\\
    &= \sum_h \widehat{y}(h) \langle \lambda_g^* x \lambda_h \delta_e, \delta_e \rangle =\sum_h \widehat{y}(h)\widehat{x}(gh^{-1})=(\widehat{x} \ast \widehat{y})(g).
    \end{align*}

This completes the proof!
\end{proof}

\begin{prop}\label{prop:fourier-facts}
    Let $G$ be a countable discrete group. Then the following statements hold:
    \begin{itemize}
        \item[(i)] $\widehat{L(G)}$ is an involutive algebra with respect to the convolution product, where the involution $\xi^*$ of $\xi:G\to\C$ is defined by
        \[
            \xi^*(g):=\overline{\xi(g^{-1})}, \qquad g\in G.
        \]
        
        \item[(ii)] The map
        \begin{align*}
            \fF_G:L(G) &\to \widehat{L(G)}\\
            x &\mapsto \widehat{x}
        \end{align*}
        is a $*$-algebra isomorphism. The map $\fF_G$ is called the \emph{Fourier transform} on $G$.
        \item[(iii)] For any $g\in G$, $\fF_G(\lambda_g) = \widehat{\lambda_g} = \delta_g$.
        \item[(iv)] For any $g\in G$ and $x \in L(G)$, we have $\fF_G(\lambda_g x) = \fF_G(x) (g^{-1}\, - )$.
        \item[(v)] For any $x \in L(G)$, we have
        \[
       x = 0 \iff  \forall g\in G,\,\, \widehat{x}(g) = 0.
        \]
        \item[(vi)] For a subgroup $H \le G$, we have $\fF_{H} = \fF_G |_{L(H)}$. (Notice that $\widehat{\lambda}_h(g)=\tau(\lambda_{g^{-1}h}) =0$ for any $g \not \in H$.) 
    \end{itemize}
\end{prop}

\begin{lem}\label{lem:fourier-trans-clasical}
    Let $G$ be a discrete abelian group. Its Pontryagin dual $\widehat{G}$ is the group of all characters $\chi: G \to \mathbb{T}$, which is a compact abelian group equipped with a normalized Haar measure.

    Let $\fF_G^c: \ell^2(G) \to L^2(\widehat{G})$ be defined as follows:
     $$\fF_G^c(\xi)(\chi) = \sum_{g \in G} \xi(g)\chi(g),$$
   for all $\xi \in \ell^2(G)$ and $\chi \in \widehat{G}$, where the sum converges in $L^2(\widehat{G})$. Then the restriction of $\fF_G^c$ to the subspace $\widehat{L(G)}$ yields a $*$-algebra isomorphism onto $L^\infty(\widehat{G})$.
\end{lem}
\begin{proof}
    For $\xi, \eta \in \widehat{L(G)}$, their multiplication is given by convolution: $(\xi * \eta)(g) = \sum_{h \in G} \xi(h)\eta(h^{-1}g)$. Then, we get
    \begin{align*}\fF_G^c(\xi * \eta)(\chi) &= \sum_{g \in G} (\xi * \eta)(g) \chi(g) \&= \sum_{g \in G} \left( \sum_{h \in G} \xi(h)\eta(h^{-1}g) \right) \chi(g).\end{align*}
    Since $\chi$ is a group homomorphism, $\chi(g) = \chi(h \cdot h^{-1}g) = \chi(h)\chi(h^{-1}g)$, we have
    \begin{align*}\fF_G^c(\xi * \eta)(\chi) 
    &= \sum_{g \in G} \sum_{h \in G} \xi(h)\chi(h) \cdot \eta(h^{-1}g)\chi(h^{-1}g) 
    \\ &= \left( \sum_{h \in G} \xi(h)\chi(h) \right) \left( \sum_{k \in G} \eta(k)\chi(k) \right) \quad (\text{letting } k = h^{-1}g) 
    \\ &= \fF_G^c(\xi)(\chi) \cdot \fF_G^c(\eta)(\chi).\end{align*}
    Thus, convolution on $G$ translates to pointwise multiplication on $\widehat{G}$.

    Since the involution in $\widehat{L(G)}$ is $\xi^*(g) = \overline{\xi(g^{-1})}$, we get
    \begin{align*}\fF_G^c(\xi^{*})(\chi) &= \sum_{g \in G} \xi^{*}(g) \chi(g) = \sum_{g \in G} \overline{\xi(g^{-1})} \chi(g).\end{align*}
    Let $h = g^{-1}$. Because $\chi$ is a character mapping into $\mathbb{T}$, $\chi(h^{-1}) = \overline{\chi(h)}$. The sum becomes
    \begin{align*}\sum_{h \in G} \overline{\xi(h)} \chi(h^{-1}) = \sum_{h \in G} \overline{\xi(h)} \,\,\, \overline{\chi(h)} = \overline{ \sum_{h \in G} \xi(h) \chi(h) } = \overline{\fF_G^c(\xi)(\chi)}.\end{align*}
    This shows that the involution translates to pointwise complex conjugation on $\widehat{G}$.

For any $x \in L(G)$, let $f = \fF_G^c(\widehat{x})$ . 
     We define the multiplication operator $M_f: L^2(\widehat{G}) \to L^2(\widehat{G})$ by $M_f(\psi) = f \cdot \psi$. To show $f \in L^\infty(\widehat{G})$, it is enough to prove $M_f$ is a bounded operator on $L^2(\widehat{G})$. Take any $\psi \in L^2(\widehat{G})$. Because $\fF_G^c$ is unitary, there exists a unique $\eta \in \ell^2(G)$ such that $\psi = \fF_G^c(\eta)$, we have
    $$M_f(\psi) = \fF_G^c(\widehat{x}) \cdot \fF_G^c(\eta) = \fF_G^c(\widehat{x} * \eta) = \fF_G^c(x(\eta))\qquad (\text{since\,}\, x(\eta) = \widehat{x} * \eta)$$
    Then, we get
    \begin{align*}\| M_f(\psi) \|_{L^2(\widehat{G})} 
    &= \| \fF_G^c(x(\eta)) \|_{L^2(\widehat{G})} \\
    &= \| x(\eta) \|_{\ell^2(G)} \quad (\text{since } \fF_G^c \text{ is unitary}) \\
&\le \| x \|_{\sB(\ell^2)} \cdot \| \eta \|_{\ell^2(G)} \quad (\text{since } x \text{ is bounded}) \\
&= \| x \|_{\sB(\ell^2)} \cdot \| \psi \|_{L^2(\widehat{G})} \quad (\text{since } \eta = (\fF_G^c)^{-1}(\psi)).
\end{align*}
This shows that for all $\psi \in L^2(\widehat{G})$,
$$\Vert{} M_f(\psi) \Vert{}_{L^2(\widehat{G})} \le \Vert{} x \Vert{} \cdot \Vert{} \psi \Vert{}_{L^2(\widehat{G})}.$$
Therefore, $M_f$ is a bounded operator on $L^2(\widehat{G})$. Hence, $f \in L^\infty(\widehat{G})$. Conversely, any $f \in L^\infty(\widehat{G})$ generates a bounded multiplication operator $M_f$, which conjugates back via $(\fF_G^c)^{-1}$ to a bounded convolution operator on $\ell^2(G)$. Thus, the map is a bijection onto $L^\infty(\widehat{G})$.
\end{proof}

\begin{prop}\label{prop:unitary-with-fourier-coef}
    Let $G$ be a countable discrete group with $|G|\ge 2$. Then there always exists a unitary $u\in L(G)$ with at least two non-zero distinct Fourier coefficients.
\end{prop}
\begin{proof}
Since $\vert{}G\vert{} \ge 2$, there exists at least one element $g \in G$ such that $g \neq e$. Consider the element $h \in L(G)$ defined by
$$h := \lambda_g + \lambda_{g^{-1}}.$$
Since $\lambda_{x}^* = \lambda_{x^{-1}}$ for any $x \in G$, we have
$$h^* = \lambda_g^* + \lambda_{g^{-1}}^* = \lambda_{g^{-1}} + \lambda_g = h.$$
Thus, $h$ is a self-adjoint element of $L(G)$. For any real number $t \in \R$, we define$$u_t := \exp(it h) = \sum_{n=0}^\infty \frac{(it)^n}{n!} h^n.$$

Since $h$ is self-adjoint, standard continuous functional calculus dictates that $u_t$ is a unitary in $L(G)$ for all $t \in \R$. Explicitly, $u_t^* = \exp(-it h^*) = \exp(-it h) = u_{-t}$, and $u_t u_{-t} = u_{-t} u_t = \lambda_e$.

We will now analyze the Fourier coefficients of $u_t$ at the identity $e$ and at $g$. Recall that the canonical trace $\tau: L(G) \to \C$ is norm-continuous, and multiplication by $\lambda_x^*$ is norm-continuous. Therefore, we can compute the Fourier coefficients by passing $\tau$ inside the norm-convergent series
$$\widehat{u_t}(x) = \tau(\lambda_{x^{-1}} u_t) = \sum_{n=0}^\infty \frac{(it)^n}{n!} \tau(\lambda_{x^{-1}} h^n), \qquad x \in G.$$
The Fourier coefficient at $x = e$
$$\widehat{u_t}(e) = \sum_{n=0}^\infty \frac{(it)^n}{n!} \tau(h^n).$$ Since the function $t \mapsto \widehat{u_t}(e)$ is continuous, at $t=0$, we have:
$$\widehat{u_0}(e) = \tau(h^0) = \tau(\lambda_e) = 1.$$
Then by continuity, there exists some $\epsilon_1 > 0$ such that for all $t \in (-\epsilon_1, \epsilon_1)$, we have $\widehat{u_t}(e) \neq 0$.

Next, let us examine the Fourier coefficient at $x = g$
$$\widehat{u_t}(g) = \tau(\lambda_{g^{-1}} u_t) = \sum_{n=0}^\infty \frac{(it)^n}{n!} \tau(\lambda_{g^{-1}} h^n).$$
Let us compute the first two terms ($n=0$ and $n=1$) of this series: For $n = 0$:
$$\tau(\lambda_{g^{-1}} h^0) = \tau(\lambda_{g^{-1}}).$$
Since $g \neq e$, we have $g^{-1} \neq e$, and thus $\tau(\lambda_{g^{-1}}) = 0$. For $n = 1$
\begin{align*}
\tau(\lambda_{g^{-1}} h) &= \tau(\lambda_{g^{-1}} (\lambda_g + \lambda_{g^{-1}})) \\ &= \tau(\lambda_e + \lambda_{g^{-2}}) \\ &= \tau(\lambda_e) + \tau(\lambda_{g^{-2}}) \\ &= 1 + \delta_{g^2, e},
\end{align*}
where $\delta_{g^2, e} = 1$ if $g^2 = e$, and $0$ otherwise. In either case, the quantity $c := 1 + \delta_{g^2, e}$ is strictly positive (it is either $1$ or $2$). Substituting these back into the series for $\widehat{u_t}(g)$, we get
$$\widehat{u_t}(g) = 0 + it \cdot c + \sum_{n=2}^\infty \frac{(it)^n}{n!} \tau(\lambda_{g^{-1}} h^n).$$
For $t \neq 0$, we get
$$\frac{\widehat{u_t}(g)}{t} = ic + t \sum_{n=2}^\infty \frac{i^n t^{n-2}}{n!} \tau(\lambda_{g^{-1}} h^n).$$

Since the remaining series converges and defines a continuous function of $t$, taking the limit as $t \to 0$ yields
$$\lim_{t \to 0} \frac{\widehat{u_t}(g)}{t} = ic \neq 0.$$Because this limit is non-zero, $\widehat{u_t}(g)$ cannot be identically zero on any open interval around $0$. Thus, there exists some $\epsilon_2 > 0$ such that for all $t \in (0, \epsilon_2)$, we have $\widehat{u_t}(g) \neq 0$. Let $t_0$ be any real number such that $0 < t_0 < \min(\epsilon_1, \epsilon_2)$. For this $t_0$, the element $u := u_{t_0} \in L(G)$ is a unitary and satisfies $\widehat{u}(e) \neq 0$ and $\widehat{u}(g) \neq 0$.

This completes the proof. 
\end{proof}

\begin{lem}\label{lem:non-subgr-alg}
    Let $\mF_2 = \langle a, b\rangle$ be the free group on two generators and $u\in L(\langle b \rangle)$ be a unitary as in the Proposition \ref{prop:unitary-with-fourier-coef}. Then $u L(\langle a \rangle)u^*$ is not of the form $L(H)$ for any subgroup $H$ of $\mF_2$.
\end{lem}
\begin{proof}
    On contrary, suppose there exists a subgroup $H\leq\mF_2$ such that
    $$u L(\langle a \rangle) u^* = L(H).$$
    
    Since $u \in L(\langle b \rangle)$ has at least two non-zero Fourier coefficients, by Proposition \ref{prop:fourier-facts}-(vi) its Fourier transform $\widehat{u}$ is supported on $\langle b \rangle$. Thus, there exist distinct integers $j_1, j_2 \in \Z$ such that $\widehat{u}(b^{j_1}) \neq 0$ and $\widehat{u}(b^{j_2}) \neq 0$. Consider the element $x = u \lambda_a u^*$. Since $\lambda_a \in L(\langle a \rangle)$, it follows that $x \in u L(\langle a \rangle) u^* = L(H)$. By Lemma \ref{lem:mult_is_convolution}, we get 
    $$\widehat{x} = \widehat{u \lambda_a u^*} = \widehat{u} * \widehat{\lambda_a} * \widehat{u^*}.$$
    
    Then, for any $g\in \mF_2$, we have
    $$\widehat{x}(g) = \sum_{y, z \in \mF_2} \widehat{u}(y) \widehat{\lambda_a}(y^{-1} z) \widehat{u^*}(z^{-1} g).$$
    
    By Proposition \ref{prop:fourier-facts}-(iii), $\widehat{\lambda_a} = \delta_a$. Thus, the term $\widehat{\lambda_a}(y^{-1} z)$ is non-zero if and only if $y^{-1} z = a$, which means $z = ya$. The sum simplifies to
    $$\widehat{x}(g) = \sum_{y \in \mF_2} \widehat{u}(y) \widehat{u^*}(a^{-1} y^{-1} g).$$
    
    Letting $w = a^{-1} y^{-1} g$, this sum is exactly over all pairs $(y, w)$ such that $y a w = g$. Furthermore, by Proposition \ref{prop:fourier-facts}-(i),(ii), $\fF_G$ is a $*$-algebra isomorphism, so $\widehat{u^*} = (\widehat{u})^*$. This means
    $$\widehat{u^*}(w) = (\widehat{u})^*(w) = \overline{\widehat{u}(w^{-1})}.$$
    Since $\widehat{u}$ is supported entirely on $\langle b \rangle$, $\widehat{u}(y)$ is non-zero only if $y = b^m$ for some $m \in \Z$. Similarly, $\widehat{u^*}(w)$ is non-zero only if $w^{-1} \in \langle b \rangle$, which means $w = b^n$ for some $n \in \Z$. Now, let us evaluate $\widehat{x}$ at the group element $g = b^{j_1} a b^{-j_2}$. The non-zero terms in our convolution sum require $b^m a b^n = b^{j_1} a b^{-j_2}$. Because $\mF_2$ is a free group on generators $a$ and $b$, elements have a unique reduced word representation. This forces $m = j_1$ and $n = -j_2$. Hence, there is exactly one non-zero term in the sum
    $$\widehat{x}(b^{j_1} a b^{-j_2}) = \widehat{u}(b^{j_1}) \widehat{u^*}(b^{-j_2}) = \widehat{u}(b^{j_1}) \overline{\widehat{u}(b^{j_2})}.$$
    
    Because both $\widehat{u}(b^{j_1})$ and $\widehat{u}(b^{j_2})$ are non-zero, this product is non-zero. By Proposition \ref{prop:fourier-facts}-(vi), since $x \in L(H)$, its Fourier transform $\widehat{x}$ is supported on $H$. Therefore, the fact that $\widehat{x}(b^{j_1} a b^{-j_2}) \neq 0$ implies
    $$h_1 := b^{j_1} a b^{-j_2} \in H.$$
    
    By identical reasoning, setting $g = b^{j_1} a b^{-j_1}$, the unique solution for $b^m a b^n = b^{j_1} a b^{-j_1}$ is $m = j_1, n = -j_1$, giving
    $$\widehat{x}(b^{j_1} a b^{-j_1}) = \widehat{u}(b^{j_1}) \overline{\widehat{u}(b^{j_1})} = \vert{}\widehat{u}(b^{j_1})\vert{}^2 \neq 0.$$
    Thus, we also have
    $$h_2 := b^{j_1} a b^{-j_1} \in H.$$
    Because $H$ is a subgroup, $h_2^{-1} h_1 \in H$. Also
    $$h_2^{-1} h_1 = (b^{j_1} a b^{-j_1})^{-1} (b^{j_1} a b^{-j_2}) = (b^{j_1} a^{-1} b^{-j_1}) (b^{j_1} a b^{-j_2}) = b^{j_1} a^{-1} a b^{-j_2} = b^{j_1 - j_2}.$$
    Let $d = j_1 - j_2$. Since $j_1 \neq j_2$, $d \neq 0$, we have  $\lambda_{b^d} \in L(H)$. By our initial assumption, $L(H) = u L(\langle a \rangle) u^*$. This means there exists some $y \in L(\langle a \rangle)$ such that
    $$\lambda_{b^d} = u y u^*.$$
    Multiplying on the left by $u^*$ and on the right by $u$, we obtain
  $  y = u^* \lambda_{b^d} u.$
    Notice that both $u$ and $\lambda_{b^d}$ belong to $L(\langle b \rangle)$. Because the subgroup $\langle b \rangle \cong \Z$ is abelian, $L(\langle b \rangle)$ is commutative. Thus, $u$ and $\lambda_{b^d}$ commute, which gives
    $$y = \lambda_{b^d} u^* u = \lambda_{b^d}.$$
    
    But $y$ was chosen from $L(\langle a \rangle)$. Therefore, $\lambda_{b^d} \in L(\langle a \rangle)$. By Proposition \ref{prop:fourier-facts}-(vi), any element belonging to $L(\langle a \rangle)$ must have its Fourier transform supported entirely on $\langle a \rangle$. However, by Proposition \ref{prop:fourier-facts}-(iii), the Fourier transform of $\lambda_{b^d}$ is $\delta_{b^d}$. Since $d \neq 0$ and $a, b$ are free generators, the element $b^d$ does not belong to the subgroup $\langle a \rangle$. Thus, $\delta_{b^d}$ is not supported on $\langle a \rangle$, which is a contradiction.
    
    Consequently, our assumption was false, and $u L(\langle a \rangle) u^*$ is not of the form $L(H)$ for any subgroup $H \leq \mF_2$
\end{proof}

\begin{prop}\label{prop:unitary_conj_not_subgp}
    There exists unitary $u \in L(\mF_2)$ such that $uL(\langle a \rangle)u^*$ is not of the form $L(H)$ for any subgroup $H$ of $\mF_2$. 
\end{prop}
\begin{proof}
    Follows from Lemma \ref{lem:non-subgr-alg}.
\end{proof}

In particular, the Proposition \ref{prop:unitary_conj_not_subgp} shows that there exist masas in \(L(G)\) that are not associated with any subgroup of \(G\); that is, they cannot be expressed in the form \(L(H)\) for any subgroup \(H\leq G\).

\begin{prop}\label{prop:haar-uni-non-trivial}
    Let $G$ be a countable discrete group with $|G|\ge 4$. Then there always exists a unitary $u\in L(G)$ with at least two non-zero distinct Fourier coefficients such that $\tau(u)=0$
\end{prop}
\begin{proof}
    We only prove this in the case when $G$ contains an element of infinite order.

Let $a \in G$ have infinite order, so $\langle a \rangle \cong \Z$. By the Fourier transform (see Lemma \ref{lem:fourier-trans-clasical}), $L(\langle a \rangle) \cong L^\infty(\mathbb{T})$, mapping $\lambda_a$ to the coordinate function $z$.

Consider the Blaschke product analytic on the unit disk
$$f(z) = z \frac{z - 1/2}{1 - z/2}.$$
For any $z \in \mathbb{T}$, $\vert{}f(z)\vert{} = 1$, which implies the operator $u = f(\lambda_a)$ is a unitary in $L(\langle a \rangle)$. Expanding $f(z)$ as a Taylor series around 0 (which coincides with its Fourier series on $\mathbb{T}$):
$$f(z) = z \left(z - \frac{1}{2}\right) \sum_{n=0}^\infty \left(\frac{z}{2}\right)^n = -\frac{1}{2} z + \frac{3}{4} z^2 + \frac{3}{8} z^3 + \dots$$
Since the constant term is $f(0) = 0$, we have $\widehat{u}(e) = 0$. The coefficients for $z, z^2, z^3, \dots$ correspond to the Fourier coefficients of $u$ at $a, a^2, a^3, \dots$ which are non-zero. Thus $u$ has infinitely many non-zero Fourier coefficients.

\end{proof}

\begin{lem}\label{lem:unitary_conjugate}
    Let $G$ be a countable discrete group and $H$ be its subgroup. Then for any $x\in L(H)$ and unitary $u\in L(G)$ we have
    \[
    \widehat{u^*xu}(g) = \sum_{s \in G} \sum_{h \in H} \overline{\widehat{u}(s)} \widehat{x}(h) \widehat{u}(h^{-1}sg).
    \]
\end{lem}
\begin{proof}
By Lemma \ref{lem:fourier_coeff}-(ii) the $k$-th Fourier coefficient of $\eta(u^{*})$ is  $\overline{\widehat{u}(k^{-1})}.$
As $x \in L(H)$, the Fourier expansion of $\eta(x)$ is supported entirely on $H$, i.e.,
$$\widehat{x}(m) = 0 \quad \text{for all } m \notin H.$$
Thus for any $t\in G$, we have 
$$(\widehat{x} * \widehat{u})(t) = \sum_{m \in G} \widehat{x}(m) \widehat{u}(m^{-1}t)= \sum_{h \in H} \widehat{x}(h) \widehat{u}(h^{-1}t) .$$
Hence, for any $g\in G$ by Lemma \ref{lem:mult_is_convolution}, we get
$$\widehat{u^*xu}(g) = (\widehat{u^*} * (\widehat{x} * \widehat{u}))(g) = \sum_{k \in G} \widehat{u^*}(k) (\widehat{x} * \hat{u})(k^{-1}g)= \sum_{k \in G} \overline{\widehat{u}(k^{-1})} \sum_{h \in H} \widehat{x}(h) \widehat{u}(h^{-1}k^{-1}g).$$
Let $s = k^{-1}$. Then, we get
$$\widehat{u^*xu}(g) = \sum_{s \in G} \sum_{h \in H} \overline{\widehat{u}(s)} \widehat{x}(h) \widehat{u}(h^{-1}sg).$$
\end{proof}

The remainder of this section is devoted to computing the distance between the masa $L(\langle a\rangle)$ and certain of its unitary conjugates inside the factor $L(\mF_2)$. We begin with some preliminary lemmas.

\begin{lem}\label{lem:non-zero-fourier-coef}
    Let $G$ be a countable discrete group. Then for any $x \in L(G)$ we have
    \[
    \| \eta(x) - \tau(x) \eta(\lambda_e) \|_\tau^2 = \sum_{g \ne e} |\widehat{x}(g)|^2.
    \]
\end{lem}
\begin{proof}
    Indeed,
    \[
    \| \eta(x) - \tau(x) \eta(\lambda_e) \|_\tau^2 = \|\eta(x - \tau(x) \lambda_e)\|_\tau^2 = \tau((x-\tau(x))^*(x-\tau(x))) = \|\eta(x)\|_\tau^2 - |\tau(x)|^2 = \sum_{g\ne e} |\widehat{x}(g)|^2.
    \]
\end{proof}

\begin{lem}\label{lem:norm_value}
   Let $\mF_2= \langle a, b \rangle$ be the free group on two generators and the subgroup $H= \langle a \rangle$. Consider the inclusion of finite von Neumann algebras $L(H) \subseteq L(\mF_2)$ with the unique $\tau$-preserving conditional expectation $E_H$ from $ L(\mF_2)$ onto $L(H)$. Suppose $u\in L(\mF_2)$ be a unitary and $F:=\operatorname{Ad_{u}}\circ E_H \circ \operatorname{Ad_{u^*}}$ is the unique $\tau$-preserving conditional expectation from $L(\mF_2)$ onto $uL(H)u^*$. Then for any $x\in \ball{L(H)}$, we have 
   \[
   \left\Vert{} \eta(x)- \eta(F(x)) \right\Vert{}_{\tau}  
   = \sqrt{\|\eta(x)\|_{\tau}^2 - \|\eta(F(x))\|_{\tau}^2} =
   \sqrt{ \sum_{h \in H} |\widehat{x}(h)|^2  - \sum_{g \in H} |\widehat{u^*xu}(g)|^2 }.
   \]
\end{lem}
\begin{proof}
For any $x\in L(\mathbb{F}_2)$, by the Pythagorean theorem,
\begin{align*}
\Vert{}\eta(x)\Vert{}_{\tau}^2 &= 
\Vert{}e_{uL(H)u^*}(\eta(x))\Vert{}_{\tau}^2 + \Vert{}(1-e_{uL(H)u^*})(\eta(x))\Vert{}_{\tau}^2\\
&=\Vert{}\eta(F(x))\Vert{}_{\tau}^2 + \Vert{}\eta(x)-\eta(F(x))\Vert{}_{\tau}^2 \qquad \text{(By eqn. \ref{eqn:jones-proj})}
\end{align*}
Consequently,
$
\Vert{}\eta(x)-\eta(F(x))\Vert{}_{\tau}^2= \Vert{}\eta(x)\Vert{}_{\tau}^2- \Vert{}\eta(F(x))\Vert{}_{\tau}^2.
$
As $u$ is a unitary, conjugation by $u$ is an isometry in the $\tau$-norm. Therefore, we have 
$$\Vert{}\eta(F(x))\Vert{}_{\tau} = \Vert{}\eta(u E_{H}(u^* x u) u^*)\Vert{}_{\tau} = \Vert{}\eta(E_{H}(u^* x u))\Vert{}_{\tau}.$$
We know that, $ \Vert{}\eta(x)\Vert{}_{\tau}^2= \sum_{h \in H} |\widehat{x}(h)|^2$,
and for any, $g\in H$,  we get  
\[
\widehat{E_{H}(u^* x u)}(g)= \tau\left(\lambda_{g}^*E_{H}(u^* x u)\right)= \tau\left(E_{H}(\lambda_{g}^*u^*xu)\right)= \tau(\lambda_{g}^*u^*xu)= \widehat{u^*xu}(g).
\]
Hence, 
\[
\Vert{}\eta(E_{H}(u^* x u))\Vert{}_{\tau}^2=
\sum_{g\in H}|\widehat{u^*xu}(g)|^2 .
\]
Therefore, we get 
\[
 \left\Vert{} \eta(x)- \eta(F(x)) \right\Vert{}_{\tau} ^2= \sum_{h \in H} |\widehat{x}(h)|^2 - \sum_{g\in H}|\widehat{u^*xu}(g)|^2. 
\]
This completes the proof!
\end{proof}

\begin{lem}\label{lem:sup_one_side}
    Let $u \in L(\langle b \rangle)$ be a unitary. Then 
    \[
    \sup_{x\in \ball{L(H)}} \left\Vert{} \eta(x)-\eta(F(x))\right\Vert{}_{\tau}= \sqrt{1 - \vert{}\tau(u)\vert{}^4}.
    \]
\end{lem}
\begin{proof}
Let $x\in \ball{L(H)}$. By Lemma \ref{lem:norm_value}, we have 
\[
\left\Vert{} \eta(x)- \eta(F(x)) \right\Vert{}_{\tau} ^2= 
    \sum_{h \in H} |\widehat{x}(h)|^2  - \sum_{g \in H} |\widehat{u^*xu}(g)|^2
\]
Since $H=\gen{a}$, by Lemma \ref{lem:unitary_conjugate} for any $m\in \mathbb{Z}$, we get 
$$\widehat{u^*xu}(a^m) = \sum_{s \in \mathbb{F}_2} \sum_{h \in \langle a \rangle} \overline{\widehat{u}(s)} \widehat{x}(h) \widehat{u}(h^{-1}s a^m).$$
Because $u \in L(\langle b \rangle)$, its Fourier coefficient $\widehat{u}(w)$ is non-zero if and only if $w \in \langle b \rangle$. Therefore, the first term $\overline{\widehat{u}(s)}$ forces $s = b^j$ for some $j \in \mathbb{Z}$. Thus we get,
$$\widehat{uxu^*}(a^m) = \sum_{j \in \mathbb{Z}} \sum_{n \in \mathbb{Z}} \overline{\widehat{u}(b^j)} \widehat{x}(a^n) \widehat{u}(a^{-n} b^j a^m)$$
For the final term $\widehat{u}(a^{-n} b^j a^m)$ to be non-zero, the argument $a^{-n} b^j a^m$ must belong to $\langle b \rangle$. Thus we have the following two cases:\\
\noindent \textbf{Case (i).} When $m = 0$, the argument becomes $a^{-n} b^j$.
For $a^{-n} b^j$ to belong to $\langle b \rangle$, all $a$ generators must vanish, which means we must have $n = 0$.
Thus, we get
$$\widehat{u^*xu}(e) = \sum_{j \in \mathbb{Z}} \overline{\widehat{u}(b^j)} \widehat{x}(e) \widehat{u}(b^j)=  \widehat{x}(e) \sum_{j \in \mathbb{Z}} \vert{}\widehat{u}(b^j)\vert{}^2= \widehat{x}(e) \Vert{}\eta(u)\Vert{}_{\tau}^2= \widehat{x}(e).$$
\noindent \textbf{Case (ii).} When $m \neq 0$, we need the word $a^{-n} b^j a^m$ to belong to $\langle b \rangle$. We analyze $j$: 
If $j \neq 0$: The $b^j$ acts as a block preventing any cancellation between $a^{-n}$ and $a^m$. For this word to be in $\langle b \rangle$, we would need both $n = 0$ and $m = 0$. Since we assumed $m \neq 0$, this is impossible. Thus, all terms with $j \neq 0$ are exactly zero. If $j = 0$: The argument simplifies to $a^{-n} a^m = a^{m-n}$. For this to be in $\langle b \rangle$, it must equal the identity $e$, which implies $m - n = 0$, or $n = m$. Therefore, out of the entire double sum, only a single term survives: the term where $j = 0$ and $n = m$. Then, we get 
$$\widehat{u^*xu}(a^m) = \overline{\widehat{u}(b^0)} \widehat{x}(a^m) \widehat{u}(a^{-m} b^0 a^m)= \overline{\widehat{u}(e)} \widehat{x}(a^m) \widehat{u}(e).$$
Since $\widehat{u}(e) = \tau(u)$, this simplifies to,
$$\widehat{u^*xu}(a^m) = \vert{}\tau(u)\vert{}^2 \widehat{x}(a^m).$$
Then, we have $$\Vert{}\eta(F(x))\Vert{}_{\tau}^2 = \sum_{g \in H} \vert{}\widehat{u^*xu}(g)\vert{}^2 = \vert{}\widehat{x}(e)\vert{}^2 + \vert{}\tau(u)\vert{}^4 \sum_{m \neq 0} \vert{}\widehat{x}(a^m)\vert{}^2.$$
Subtracting this from $\Vert{}\eta(x)\Vert{}_{\tau}^2 = \vert{}\widehat{x}(e)\vert{}^2 + \sum_{m \neq 0} \vert{}\widehat{x}(a^m)\vert{}^2$ yields
$$\Vert{}\eta(x)\Vert{}_{\tau}^2 - \Vert{}\eta(F(x))\Vert{}_{\tau}^2 = (1 - \vert{}\tau(u)\vert{}^4) \sum_{m \neq 0} \vert{}\widehat{x}(a^m)\vert{}^2.$$
Notice that the sum on the right side is exactly $\Vert{}\eta(x - \tau(x)\lambda_e)\Vert{}_{\tau}^2$ (by Lemma \ref{lem:non-zero-fourier-coef}). Thus
$$\left\Vert \eta(x-F(x)) \right\Vert_{\tau} = \sqrt{1 - \vert{}\tau(u)\vert{}^4} \; \Vert{}\eta(x - \tau(x)\lambda_e)\Vert{}_{\tau}.$$

Because $\sqrt{1 - \vert{}\tau(u)\vert{}^4}$ is fixed for a given unitary $u$. Thus to find the supremum it is enough to maximize the term, $\Vert{}\eta(x - \tau(x)\lambda_e)\Vert{}_{\tau}
$
for all $x\in \ball{L(H)}$.
Therefore for any $x\in \ball{L(H)}$, we get
$$\Vert{}\eta(x - \tau(x)\lambda_e)\Vert{}_{\tau}^2 = \Vert{}\eta(x)\Vert{}_{\tau}^2 - \vert{}\tau(x)\vert{}^2 \le \Vert{}\eta(x)\Vert{}_{\tau}^2 \le \op{x}^2 \le 1.$$
Also, for any unitary $v$ in $L(H)$ with zero trace (for example take $v = \lambda_h$ for $h \ne e$), we have 
$$
\Vert{}\eta(v - \tau(v)\lambda_e)\Vert{}_{\tau}^2=1.
$$
Thus, we get
$$\sup_{x \in \ball{L(H)}}\left\Vert \eta(x-F(x)) \right\Vert_{\tau} = \sqrt{1 - \vert{}\tau(u)\vert{}^4} \left(\sup_{x \in \ball{L(H)}}\Vert{}\eta(x - \tau(x)\lambda_e)\Vert{}_{\tau}\right)= \sqrt{1 - \vert{}\tau(u)\vert{}^4}.$$
\end{proof}

\begin{lem}\label{lem:sup_other_side}
 Let  $u \in L(\langle b \rangle)$ be a unitary. Then 
$$\sup_{y\in \ball{u L(H)u^*}}\|\eta(y)-\eta(E_{H}(y))\|_{\tau}= \sqrt{1 - \vert{}\tau(u)\vert{}^4}$$   
\end{lem}
\begin{proof}
    Suppose $y = uxu^*$ for some $x \in \ball{L(H)}$. By eqn. \ref{eqn:jones-proj} and the Pythgorean theorem, we get
 \[   \Vert{}\eta(uxu^*)-\eta(E_H(uxu^*))\Vert{}_{\tau}^2= \Vert{}\eta(uxu^*)\Vert{}_{\tau}^2- \Vert{}\eta(E_H(uxu^*))\Vert{}_{\tau}^2=\Vert{}\eta(x)\Vert{}_{\tau}^2- \Vert{}\eta(E_H(uxu^*))\Vert{}_{\tau}^2.\] 
 For any $g\in H$, we get
 $$\widehat{E_{H}(uxu^*)}(g)= \tau\left(\lambda_{g}^*(E_H(uxu^*))\right)= \tau\left(E_H(\lambda_{g}^*uxu^*)\right)=\tau(\lambda_{g}^*uxu^*)= \widehat{uxu^*}(g).
 $$
By Lemma \ref{lem:unitary_conjugate}, we get 
\[
\widehat{uxu^*}(a^m)= \sum_{k\in \mF_2}\sum_{h\in H}\widehat{u}(k)\widehat{x}(h)\overline{\widehat{u}(a^{-m}kh)}
\]
Because $u \in L(\langle b \rangle)$, its Fourier coefficient $\widehat{u}(w)$ is non-zero if and only if $w \in \langle b \rangle$. Therefore, the first term $\overline{\widehat{u}(s)}$ forces $k = b^j$ for some $j \in \mathbb{Z}$. Thus, we get
$$\widehat{uxu^*}(a^m) = \sum_{j \in \mathbb{Z}} \sum_{n \in \mathbb{Z}} \widehat{u}(b^j) \widehat{x}(a^n) \overline{\widehat{u}(a^{-m} b^j a^n)}$$
For the final term $\widehat{u}(a^{-m} b^j a^n)$ to be non-zero, the argument $a^{-m} b^j a^n$ must belong to $\langle b \rangle$. Thus we have the following two cases:\\
\noindent \textbf{Case (i).} When $m = 0$, the argument becomes $b^ja^n$.
For $b^ja^n$ to belong to $\langle b \rangle$, all $a$ generators must vanish, which means we must have $n = 0$.
Thus, we get
$$\widehat{uxu^*}(e) = \sum_{j \in \mathbb{Z}} \widehat{u}(b^j) \widehat{x}(e) \overline{\widehat{u}(b^j)}=  \widehat{x}(e) \sum_{j \in \mathbb{Z}} \vert{}\widehat{u}(b^j)\vert{}^2= \widehat{x}(e) \Vert{}\eta(u)\Vert{}_{\tau}^2= \widehat{x}(e).$$
\noindent \textbf{Case (ii).} When $m \neq 0$, we need the word $a^{-m} b^j a^n$ to belong to $\langle b \rangle$. We analyze $j$: 
If $j \neq 0$: The $b^j$ acts as a block preventing any cancellation between $a^{-m}$ and $a^n$. For this word to be in $\langle b \rangle$, we would need both $n = 0$ and $m = 0$. Since we assumed $m \neq 0$, this is impossible. Thus, all terms with $j \neq 0$ are exactly zero. If $j = 0$: The argument simplifies to $a^{-m} a^n = a^{n-m}$. For this to be in $\langle b \rangle$, it must equal the identity $e$, which implies $n - m = 0$, or $n = m$. Therefore, out of the entire double sum, only a single term survives: the term where $j = 0$ and $n = m$. Then, we get 
$$\widehat{uxu^*}(a^m) = \widehat{u}(b^0) \widehat{x}(a^m) \overline{\widehat{u}(a^{-m} b^0 a^n)}= \widehat{u}(e) \widehat{x}(a^m) \overline{\widehat{u}(e)}.$$
Since $\widehat{u}(e) = \tau(u)$, this simplifies to,
$$\widehat{uxu^*}(a^m) = \vert{}\tau(u)\vert{}^2 \widehat{x}(a^m).$$
Then, we have $$\Vert{}\eta(E_{H}(uxu^*))\Vert{}_{\tau}^2 = \sum_{g \in H} \vert{}\widehat{uxu^*}(g)\vert{}^2 = \vert{}\widehat{x}(e)\vert{}^2 + \vert{}\tau(u)\vert{}^4 \sum_{m \neq 0} \vert{}\widehat{x}(a^m)\vert{}^2.$$
Subtracting this from $\Vert{}\eta(x)\Vert{}_{\tau}^2 = \vert{}\widehat{x}(e)\vert{}^2 + \sum_{m \neq 0} \vert{}\widehat{x}(a^m)\vert{}^2$ yields
$$\Vert{}\eta(x)\Vert{}_{\tau}^2 - \Vert{}\eta(E_{H}(uxu^*))\Vert{}_{\tau}^2 = (1 - \vert{}\tau(u)\vert{}^4) \sum_{m \neq 0} \vert{}\widehat{x}(a^m)\vert{}^2.$$
Notice that the sum on the right side is exactly $\Vert{}\eta(x - \tau(x)\lambda_e)\Vert{}_{\tau}^2$ (by Lemma \ref{lem:non-zero-fourier-coef}). Thus
$$\left\Vert \eta(uxu^*-E_{H}(uxu^*)) \right\Vert_{\tau} = \sqrt{1 - \vert{}\tau(u)\vert{}^4} \; \Vert{}\eta(x - \tau(x)\lambda_e)\Vert{}_{\tau}.$$

Because $\sqrt{1 - \vert{}\tau(u)\vert{}^4}$ is fixed for a given unitary $u$. Thus to find the supremum it is enough to maximize the term, $\Vert{}\eta(x - \tau(x)\lambda_e)\Vert{}_{\tau}
$
for all $x\in \ball{L(H)}$.
Therefore for any $x\in \ball{L(H)}$, we get
$$\Vert{}\eta(x - \tau(x)\lambda_e)\Vert{}_{\tau}^2 = \Vert{}\eta(x)\Vert{}_{\tau}^2 - \vert{}\tau(x)\vert{}^2 \le \Vert{}\eta(x)\Vert{}_{\tau}^2 \le \op{x}^2 \le 1.$$
Also, for any unitary $v$ in $L(H)$ with zero trace (for example take $v = \lambda_h$ for $h \ne e$), we have 
$$
\Vert{}\eta(v - \tau(v)\lambda_e)\Vert{}_{\tau}^2=1.
$$
Thus, we get
$$\sup_{x \in \ball{L(H)}}\left\Vert \eta(uxu^*-E_{H}(uxu^*)) \right\Vert_{\tau} = \sqrt{1 - \vert{}\tau(u)\vert{}^4} \left(\sup_{x \in \ball{L(H)}}\Vert{}\eta(x - \tau(x)\lambda_e)\Vert{}_{\tau}\right)= \sqrt{1 - \vert{}\tau(u)\vert{}^4}.$$
\end{proof}

\begin{theorem}\label{thm:mt_is_trace}
Let $\mF_{2}= \langle a, b \rangle$ and consider a unitary $u \in L(\langle b \rangle)$. Then 
    \[
    \dmt (L(\langle a \rangle), u L(\langle a \rangle)u^*) = \sqrt{1- |\tau(u)|^4}.
    \]    
\end{theorem}
\begin{proof}
    The proof follows from Lemma \ref{lem:sup_one_side}, Lemma \ref{lem:sup_other_side} and Theorem \ref{thm:mashood_taylor}.
\end{proof}

\begin{cor}
    Let $\mF_{2}= \langle a, b \rangle$ and consider a unitary $u \in L(\langle b \rangle)$ such that $\tau(u) \ne 0$. Then $L(\langle a \rangle)$ is not orthogonal to  $u L(\langle a \rangle)u^*$ in the sense of Popa.
\end{cor}
\begin{proof}
    This follows from Theorem \ref{thm:mt_is_trace} and Theorem \ref{thm:diffuse-comm-sq}.
\end{proof}

\begin{remark}
    Thus there always exists unitary $u$ (in $L(\langle b \rangle)$) such that 
    \[
    \dkk( L(\langle a \rangle), u L(\langle a \rangle)u^*) = 1.
    \]

    Also it follows from the last result that $\dmt$ takes all values in $[0, 1]$.

    Let $\mF_2 = \langle a, b\rangle$ be the free group on two generators then we have for any $n\in \Z \setminus \{0\}$,
    \[
    \dkk (L(\langle a \rangle), \,\lambda_{b^n}L(\langle a \rangle) \lambda_{b^n}^*) = 1.
    \]
\end{remark}

\vspace{0.5cm}
Let $\sM$ be a diffuse von Neumann algebra equipped with a faithful normal state $\psi$. It is an easy exercise to verify that
\[
\psi(\sM_{\mathrm{proj}})=[0,1].
\]
Since $L(G)$ is diffuse whenever $G$ is an infinite discrete group, the next result follows immediately. Nevertheless, we provide an alternative proof using the Fourier theory developed above, for completeness.

\begin{lem}\label{lem:unitayr-trace-r}
    Let $G$ be an infinite discrete group. Then for any $r \in [0, 1]$ there is a unitary in $L(G)$ such that 
    \[
    \tau(u) = (1-r^2)^{1/4}. 
    \]
\end{lem}
\begin{proof}
Given that $G$ is of infinite order, then there exists an element $g\in G$ such that the order of $g$ is infinite. Consider the subgroup $H=\gen{g}$. Then the corresponding von Neumann algebra $L(H)$ is abelian and isomorphic to $L(\Z)$. Moreover, the trace $\tau_{|L(H)}$ is the faithful normal trace on $L(H)$. By the Fourier transform (see Lemma \ref{lem:fourier-trans-clasical}), we have $L(H)\cong L^{\infty}(\mathbb{T})$, and the trace $\tau_{|L(H)}$ corresponds to integration on $L^{\infty}(\mathbb{T})$. Thus, under this identification, it suffices to find a function $f$ with $|f|=1$ in $L^{\infty}(\mathbb{T})$ such that
\[
\int_{\mathbb{T}}f\,d\theta = (1-r^2)^{1/4}.
\]
Consider the arc $A= \{e^{i\theta}\in \mathbb{T}\,\, \mid\, \,0 \le \theta < \pi(1 + (1-r^2)^{1/4}) \}$ and 
define the function $f$ as
$$f(e^{i\theta}) = \begin{cases} 1 & \text{for } e^{i\theta}\in A \\ -1 & \text{for } e^{i\theta}\in \mathbb{T} \setminus A \end{cases}$$
Then clearly $|f|=1$ and the normalized length of the arc $A$ is $\frac{1+(1-r^2)^{1/4}}{2}$. Thus 
\[
\int_{\mathbb{T}}f\,d\theta = \int_{A}f\,d\theta+ \int_{ \mathbb{T} \setminus A}f\,d\theta= (1-r^2)^{1/4}.
\]
\end{proof}

\begin{cor}\label{cor:mt-attains-val}
    For any $r\in [0, 1]$ there exists a unitary $u \in L(\mF_2)$ such that
    \[
    \dmt(L(\langle a \rangle), uL(\langle a \rangle)u^*) = r.
    \]
\end{cor}
\begin{proof}
    By Lemma \ref{lem:unitayr-trace-r} there exists a unitary in $L(\gen{b})$ such that $\tau(u)= (1-r^2)^{1/4}$. Then by the Theorem \ref{thm:mashood_taylor}, we get the desired result.
\end{proof}

\begin{theorem}\label{thm:dkk-1-for-non-tiv-alg}
    Let $\mF_2 = \langle a, b\rangle$ be the free group on two generators then there exists a unitary $u \in L(\langle b \rangle)$ such that $uL(\langle a \rangle)u^*$ is not a group von Neumann algebra corresponding to any subgroup of $\mF_2$ and 
    \[
    \dkk (L(\langle a \rangle), \,uL(\langle a \rangle) u^*) = 1=\dmt (L(\langle a \rangle), u L(\langle a \rangle)u^*).
    \]
\end{theorem}
\begin{proof}
    Using Proposition \ref{prop:haar-uni-non-trivial}, choose a unitary $u\in L(\langle b \rangle)$ such that $\tau(u) = 0$ and $u$ has at least two non-zero Fourier coefficients. Since $u$ has two non-zero Fourier coefficient, by Lemma \ref{lem:non-subgr-alg} $uL(\langle a \rangle)u^* \ne L(H)$ for any $H\le G$. Now using Proposition \ref{prop:mt_and_kk} and Theorem \ref{thm:mt_is_trace} we get:
    \[
    1= \sqrt{1- |\tau(u)|^4}=\dmt (L(\langle a \rangle), u L(\langle a \rangle)u^*)  \le \dkk (L(\langle a \rangle), u L(\langle a \rangle)u^*) \le 1.
    \]
\end{proof}

    \section{Acknowledgment}
    We would like to thank Prof. Keshab Chandra Bakshi for his valuable insights and for patiently addressing our questions regarding spin model subfactors. Sumit Kumar gratefully acknowledges Prof. Keshab Chandra Bakshi for providing him with the opportunity to work at IIT Kanpur under his guidance through Project No. SPO/ANRF/MATH/2025271.

	\medskip
	
	\bibliographystyle{amsalpha} 
	\bibliography{references}

\end{document}